\documentclass{tut2016}

\input{preamble}

\RemoveLogos

\newcommand{\pb}{{\mathrm{PB}}}

\newcommand{\mcs}{{\mathrm{MCS}}}
\newcommand{\mcsk}{{\mathrm{MCS\mhyphen}k}}
\newcommand{\loc}{{\mathrm{loc}}}
\DeclareMathOperator{\EI}{EI}
\DeclareMathOperator{\dBeta}{Beta}
\newcommand{\stName}{\st}
\newrobustcmd{\st}{{\mathchoice
  {\mkern0.54mu\scalebox{0.8}{\ensuremath{\displaystyle\square}}}
  {\mkern0.54mu\scalebox{0.8}{\ensuremath{\textstyle\square}}}
  {\mkern0.54mu\scalebox{0.8}{\ensuremath{\scriptstyle\square}}}
  {\mkern0.54mu\scalebox{0.8}{\ensuremath{\scriptscriptstyle\square}}}%
}}
\newcommand{\goName}{\go}
\newrobustcmd{\go}{{\mathchoice
  {\raisebox{0.05em}{\scalebox{0.9}{\ensuremath{\displaystyle\vartriangleright}}}\mkern-1mu}
  {\raisebox{0.05em}{\scalebox{0.9}{\ensuremath{\textstyle\vartriangleright}}}\mkern-1mu}
  {\raisebox{0.05em}{\scalebox{0.9}{\ensuremath{\scriptstyle\vartriangleright}}}\mkern-1mu}
  {\raisebox{0.05em}{\scalebox{0.9}{\ensuremath{\scriptscriptstyle\vartriangleright}}}\mkern-1mu}%
}}
\newcommand{\tkName}{\mathrm{take}}
\newcommand{\tk}{\go_\tkName}
\newcommand{\opName}{\mathrm{open}}
\newcommand{\op}{\go_\opName}

\newcommand{\boxclosed}{{\boxtimes}}

\newcommand{\alt}{\alpha}
\newcommand{\dscnt}{\gamma}
\newcommand{\approxRatio}{\beta}
\newcommand{\surval}{\Gamma}

\ifinforms{
  \CHAPTERNO{\phantom{Chapter 1}}%   version used in the opening
  \DOI{}
  \TITLE{The Gittins Index: A Design Principle for Decision-Making Under Uncertainty}
  \AUBLOCK{%
    \AUTHOR{Ziv Scully}
    \AFF{%
        School of Operations Research and Information Engineering, Cornell University, Ithaca, NY, USA\\
        \EMAIL{zivscully@cornell.edu}%
    }
    \AUTHOR{Alexander Terenin}
    \AFF{%
        Center for Data Science for Enterprise and Society, Cornell University, Ithaca, NY, USA\\
        \EMAIL{avt28@cornell.edu}%
    }%
  }
  \CHAPTERHEAD{Ziv Scully and Alexander Terenin: The Gittins Index}
}{
  \title[The Gittins Index]{The Gittins Index: A Design Principle for Decision-Making Under Uncertainty}
  \author[Ziv Scully and Alexander Terenin]{%
    Ziv Scully\\
    Cornell University
    \and
    Alexander Terenin\\
    Cornell University
  }
}

\begin{document}

\newcommand{\theabstract}{%
The Gittins index is a tool that optimally solves a variety of decision-making problems involving uncertainty, including multi-armed bandit problems, minimizing mean latency in queues, and search problems like the Pandora's box model.
However, despite the above examples and later extensions thereof, the space of problems that the Gittins index can solve perfectly optimally is limited, and its definition is rather subtle compared to those of other multi-armed bandit algorithms.
As a result, the Gittins index is often regarded as being primarily a concept of theoretical importance, rather than a practical tool for solving decision-making problems.

The aim of this tutorial is to demonstrate that the Gittins index can be fruitfully applied to practical problems.
We start by giving an example-driven introduction to the Gittins index, then walk through several examples of problems it solves---some optimally, some suboptimally but still with excellent performance.
Two practical highlights in the latter category are applying the Gittins index to Bayesian optimization, and applying the Gittins index to minimizing tail latency in queues.%
}

\ifinforms{
  \ABSTRACT{\theabstract}
  \maketitle
}{
  \maketitle
  \begin{abstract}\theabstract\end{abstract}
}

\tableofcontents

\section{Introduction}
\label{sec:intro}

Making effective decisions under uncertainty is a central theme in many areas of science, engineering, and technology.
Algorithms for making decisions are therefore central to many fields---from operations research to economics and artificial intelligence, to name a few.
In many---though certainly not all---such situations, it can be appropriate to model the decision-making problem stochastically, by formalizing it as a fully specified Markov decision~process.

An important subclass of such Markov decision processes represents, loosely speaking, \emph{choosing the best option among a set of alternatives} in the face of stochastic feedback about which option is best.
This includes the class of Bayesian multi-armed bandits, but also a number of other decision problems which at first glance might appear to be of a rather different character, such as consumer choice problems arising in economics, or certain optimal scheduling problems arising in queueing theory.

An underappreciated fact about such problems is that \emph{there is a general way to solve them exactly} using a class of techniques broadly known as \emph{Gittins index theory}.
These techniques have been rediscovered in various communities, often through the lens of a surprising solution to a specific decision problem.
As an example, let us quote a widely known exchange between two colleagues as recalled by Peter Whittle \cite{gittins_multi-armed_2011} and abridged by Richard Weber \cite{weber_lectures_2016}:
\begin{quote}
A colleague of high repute asked an equally well-known colleague:
\*[---,labelsep={2pt}] What would you say if you were told that the multi-armed
bandit problem had been solved?
\* Sir, the multi-armed bandit problem is not of such a
nature that it \emph{can} be solved.
\*/
\end{quote}
Reflecting this sentiment, our first goal in this tutorial will be to illustrate how such solutions work, focusing chiefly on the \emph{key definitions}.
In doing so, we show that Gittins index theory tells us, intuitively: to choose the best option among a set of alternatives under stochastic feedback, \emph{compare each stochastic option with an equivalent deterministic option}.

A second, even more underappreciated fact about the decision problems we study is that \emph{the definitions arising from exact solutions of simple problems can also yield strong solutions for more complex problems} where there is no hope of an exact solution.
In his course on information-directed sampling, Tor Lattimore described the setup covered by Gittins index theory as the \emph{"miraculous case"} \citep{lattimore_lectures_2021}---to contrast it with more general situations where optimal policies are intractable.
While this might tempt one to conclude that the Gittins index is not an appropriate technical tool for such settings, the basic idea of comparing stochastic options with equivalent deterministic options often continues to make intuitive sense, even though it is no longer optimal.

Against this background, our second goal is to illustrate how Gittins index theory can be used as a design principle for general decision problems.
For this, we showcase problems where policies based on the Gittins index, while not optimal, are known to perform strongly---either in a theoretical or an empirical sense, covering examples from queueing theory and Bayesian~optimization.

Our tutorial begins by covering what Gittins indices are and how they arise, intuitively.
This is done by way of analyzing Pandora's box, arguably the simplest concrete example, in \cref{sec:pandora}---and then placing it into a suitable abstract framework in \cref{sec:general}.
We then survey various extensions and advanced examples.
\Cref{sec:examples} covers settings where optimality holds, while \cref{sec:beyond} covers settings where optimality does not hold, but the Gittins index still either (a)~has excellent empirical performance, (b)~satisfies an approximate or asymptotic optimality guarantee, or (c)~both.
Two particularly practically relevant examples are Bayesian optimization (\cref{sec:beyond:bayesopt}) and scheduling to minimize tail latency (\cref{sec:beyond:tail}).
We give a more detailed outline of the themes covered by the examples in \cref{sec:pandora:limitations}.
We conclude our tutorial in \cref{sec:conclusion} with a discussion of open problems.

\section{An illustrative example: Pandora's box}
\label{sec:pandora}

The easiest way to understand the Gittins index is by way of example: for this, we present the Pandora's box problem from economics, originally due to \citet{weitzman_optimal_1979}.
In the Pandora's box problem, the decision-making agent is presented with a set of $n$ \emph{boxes}, labeled $i \in \{1, \dots, n\}$.
For a set $X$, let $\calP(X)$ be the space of probability measures over $X$.
Each box is associated with two quantities:
\*[1.] A \emph{cost to open} $c_i > 0$.
\* A \emph{reward distribution} $p_i \in \calP(\bbR)$, assumed to have finite mean.
\*/
At the starting time, each box is assumed to be \emph{closed}: we conceptually imagine it to contain a reward inside which is unknown to the agent and viewed as random.
Starting from the state where all boxes are closed, at each time point, the agent is allowed to either:
\*[1.] \emph{Open a closed box} $i$ of their choosing: the agent pays a cost of $c_i$, and learns the precise value of the reward $v_i\sim p_i$ which is contained inside the box.
\* \emph{Select an open box} $i$ from the set of open boxes: the agent's decision-making process ends, and they receive the reward $v_i$ which is inside the box.
\*/
Letting $T$ be the time when the agent selects a box, and $i_t$ be the box opened resp. selected at time $t$, the agent's aim is to maximize their expected total value, which is
\[
\label{eq:pandora_objective}
\E[\Bigg]{v_{i_T} - \sum_{t=1}^{T-1} c_{i_t}}
.
\]
A very important aspect of this formulation is that, even though the agent can open multiple boxes and must pay a cost \emph{for each box}, they ultimately receive \emph{only one reward}---namely, the value in the box they selected at the very end.
This results in an explore-exploit tradeoff: should the agent pay to learn about more possible rewards they might eventually select, or are the rewards revealed already good enough?

\subsection{Pandora's box as a Markov decision process}
\label{sec:pandora:mdp}

We can formulate the Pandora's box problem (denoted $\pb$ in subscripts) as a discrete-time Markov decision process (MDP) as follows.
Define the state space to be the set of tuples
\[
S_\pb = \{(s_1, \dots, s_n) : s_i \in \{\boxclosed\}\cup\bbR\cup\{\checkmark\}\}
\]
where $s_i = \boxclosed$ represents box~$i$ being closed, $s_i\in\bbR$ represents box~$i$ being open with reward $v_i = s_i$, and $s_i = \checkmark$ represents box~$i$ being selected.
This is a finite-horizon undiscounted MDP, whose initial state is $(\boxclosed, \dots, \boxclosed)$, and terminal states are tuples $(s_1, \dots, s_n)$ with $s_i = \checkmark$ for some~$i$.
The action space is $A_\pb = \{1, \dots, n\}$, which refers to indices of boxes.
Each action $a \in A_\pb$ corresponds to either opening the respective box, or selecting it: the transition kernel's action replaces the corresponding identifier in the tuple.
For example, for a three-closed-box state and $a=1$, this occurs by
\[
(\boxclosed,\boxclosed,\boxclosed) \overset{a=1}{\mapsto} (v_1,\boxclosed,\boxclosed)
\]
where $v_1\sim p_1$ is the revealed reward.
Throughout this section, we focus on the classical variant where all boxes' rewards are independent.
The MDP's reward function, defined for non-terminal states and all actions, is
\[
r_\pb(s,a) = \begin{cases}
-c_a & s_a = \boxclosed
\\
s_a & s_a \in \bbR
\end{cases}
\]
which returns either the negated costs of a closed box, or the previously revealed reward of an open box---here and throughout, we work with MDPs that allow negative rewards.
Our expected total value defined previously is therefore equal to the MDP's value function.
Continuing our three-box example, choosing $a=1$ a second time would transition
\[
(v_1,\boxclosed,\boxclosed) \overset{a=1}{\mapsto} (\checkmark,\boxclosed,\boxclosed)
\]
with a reward of $v_1$, for an overall value of $v_1 - c_1$.

By general MDP theory, there exists an optimal policy which maximizes the agent's expected value.
At first glance, however, it is unclear how much more one can expect to say about this policy.
On the one hand, the decision problem is, at the vaguest level, clean enough that one might hope for a surprisingly straightforward solution---as occurs, in, say, the secretary problem.
On the other hand, for general Markov decision processes, there is little hope of saying much about the optimal policy.
It will turn out that, using the Gittins index, one can obtain the optimal policy analytically---producing a solution that will turn out to be both straightforward and subtle at the same time.

Before proceeding, we conclude with a few comments.
First, note that the agent's policy can be adaptive: they can---and should---use values from opened boxes to decide what action to take next.
Second, note that decision problem is not an unknown-MDP statistical learning problem: all of MDP parameters, and in particular all reward distributions $p_i$ defining the transition kernel, are known to the agent.
In this problem, therefore, learning takes place in the sense of conditional probability: the agent learns the box's value by opening it.

\subsection{Why obvious greedy policies are suboptimal}
\label{sec:pandora:example}

\newcommand{\figPandoraExample}{%
  \begin{figure}[t]
    \begin{tikzpicture}
      \begin{scope}[scale=3]
        \node (A) at (-1.7, 0.55) {Box~1: closed};
        \drawBoxClosed{A |- 0, 0}{$\begin{gathered} c_1 = 1 \\ v_1 \sim p_1 = \begin{cases} 14 & \t{w.p.} \frac{1}{2} \\ 0 & \t{w.p.} \frac{1}{2} \end{cases} \end{gathered}$};
        \node (B) at (0, 0.55) {Box~2: closed};
        \drawBoxClosed{B |- 0, 0}{$\begin{gathered} c_2 = 1 \\ v_2 \sim p_2 = \begin{cases} 18 & \t{w.p.} \frac{1}{5} \\ 0 & \t{w.p.} \frac{4}{5} \end{cases} \end{gathered}$};
        \node (C) at (1.7, 0.55) {Box~3: open};
        \drawBoxOpen{C |- 0, 0}{$v_3 = 10$};
      \end{scope}
    \end{tikzpicture}
    \caption{
      An instance of the Pandora's box problem with two closed boxes and one opened box.
      Here, we know the realized value of the opened box~3, but only the reward distributions and not the realized values of the closed boxes~1 and~2.
    }
    \label{fig:pandora_example}
  \end{figure}
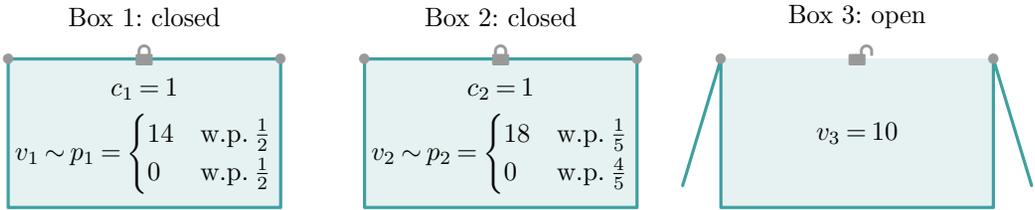%
}

Before explaining how to solve the Pandora's box problem, it is worth walking through an example that illustrates why the problem is non-trivial.
In particular, we will see that the most obvious greedy policy, consisting of choosing boxes according to the expected difference between rewards and costs, is suboptimal.
We will then see that the second-most obvious greedy policy, specifically one-step lookahead, is also suboptimal.
This might make one think greedy policies are not the right approach: remarkably, however, the optimal policy \emph{will} turn out to be a greedy policy---albeit with respect to a more subtle objective compared to the above two policies.

\figPandoraExample

Consider the three-box scenario shown in \cref{fig:pandora_example}.
In the situation represented by this state, we have already opened box~3, so we take its realized value $v_3 = 10$ to be fixed throughout this example, and ignore its cost because it has been paid already.
To decide what to do next, we have to answer two~questions:
\*[(a)] \label{item:pandora_example:should_stop}Is it worth opening another box, or should we stop now?
\* \label{item:pandora_example:which_box}If we open another box, should we open box~1 or box~2?
\*/

For \cref{item:pandora_example:should_stop}, a naïve idea is to think about how valuable each closed box~$i$ would be if it were the only box we had, which is simply $\E{v_i} -c_i$.
We compute
\[
  \E{v_1} - c_1 &= \tfrac{1}{2} \cdot 14 - 1 = 6
   \\
  \E{v_2} - c_2 &= \tfrac{1}{5} \cdot 18 - 1= 2.6
  .
\]
Both of these are much less than $v_3 = 10$, so it may seem like boxes~1 and~2 are less valuable than box~3, which might suggest we should stop now.

However, the above computation fails to account for a crucial fact: \emph{box~3's reward remains available even if we open another box}.
That is, if we were to open box~$i$, and then stop, we can receive a reward of either $v_i$ or $v_3$ when we stop, depending on our selection, and not simply~$v_i$.
We assume henceforth that we always select the best open box: this means our reward is $\max(v_i,v_3)$.\footnote{Recall again that we can ultimately only receive the reward from one box: this is why we receive the maximum of $v_i$ and $v_3$, and not their sum.}
Computing the expected value of opening each box and then stopping, we find
\[
  \label{eq:pandora_example_lookahead_1}
  \E{\max(v_1, v_3)} - c_1 &= \tfrac{1}{2} \cdot 14 + \tfrac{1}{2} \cdot 10 - 1 = 11
   \\
  \label{eq:pandora_example_lookahead_2}
\E{\max(v_2, v_3)} - c_2 &= \tfrac{1}{5} \cdot 18 + \tfrac{4}{5} \cdot 10 - 1 = 10.6
  .
\]
Both of these are greater than~$10$, so we conclude that opening either closed box and then stopping is better than stopping now---thus, we should definitely open at least one more~box.

Having decided to open a box, we move on to \cref{item:pandora_example:which_box}: which box should we open?
As a first step towards answering this, it will help to review \cref{eq:pandora_example_lookahead_1, eq:pandora_example_lookahead_2} from a different perspective.
Let the \emph{expected improvement} of box~$i$ over a baseline~$\alt$ be
\[
  \label{eq:pandora_ei}
  \EI_i(\alt) = \E{\max(v_i - \alt, 0)} - c_i
  .
\]
That is, if the current best value (among open boxes) is~$\alt$, then $\EI_i(\alt)$ is the expected amount we gain by opening box~$i$ instead of stopping now, accounting for both the cost of opening the box $-c_i$ and increasing the best value from $\alt$ to $\max(v_i, \alt)$.
The expected value of opening box~$i$ and then stopping is thus $\alt + \EI_i(\alt)$.
We can recognize \cref{eq:pandora_example_lookahead_1, eq:pandora_example_lookahead_2} as $\alt + \EI_i(\alt)$ values with baseline $\alt = v_3 = 10$:
\[
  \EI_1(v_3) &= \tfrac{1}{2} \cdot (14 - 10) - 1 = 1
\\
  \EI_2(v_3) &= \tfrac{1}{5} \cdot (18 - 10) - 1 = 0.6
  .
\]
So, box~1 has better expected improvement than box~2, which might suggest we should open box~1.
In fact, if we were only allowed to open one more box, then this computation shows that opening box~1 would be best.

Of course, we are not limited to opening just one more box.
In particular, if we open a box but find its realized value is low (say,~$0$), then $v_3 = 10$ would still be our best value.
In this case, the remaining closed box would still have positive expected improvement, so it would make sense to open it.
This reasoning suggests the following policy:
\* We first open one of the closed boxes.
\* If the realized value is high (say, $14$ or $18$), we stop.
\* If the realized value is low (say,~$0$), we open the other closed box.
\*/
We can carry this out starting by opening either box~1, which we call \emph{policy~P1}, or box~2, which we call \emph{policy P2}.
Our expected value under P1 is
\[
  \label{eq:pandora_example_brute_force_1}
  \E[\big]{\1(v_1 = 14) \cdot 14 + \1(v_1 = 0) \cdot (\max(v_2, v_3)-c_2)} - c_1
  = 11.3
\]
which is better than opening box~1 then stopping.
But our expected value under P2 is
\[
  \label{eq:pandora_example_brute_force_2}
  \E[\big]{\1(v_2 = 18) \cdot 18 + \1(v_2 = 0) \cdot (\max(v_1, v_3)-c_1)} - c_2
  = 11.4
\]
which is even better.
Thus, even though box~1 has higher expected improvement than box~2, policy~P2 outperforms~P1.
In fact, with a little bit more casework, one can show P2 is~optimal!

\subsubsection{Connecting expected improvement to one-step lookahead}

There is a connection between expected improvement and the one-step lookahead policy.
Specifically, the \emph{one-step lookahead} policy behaves as follows at every time step:
\* Let $v_{\max}$ be the maximum realized value among opened boxes.
\* If $\max_i \EI_i(v_{\max}) < 0$, namely if every box's expected improvement is negative, then stop.
\* Otherwise, open $\argmax_i \EI_i(v_{\max})$, namely the box of greatest expected improvement.
\*/
Policies like this are a standard approach used for constructing Bayesian optimization algorithms for black-box optimization, where they give rise to the \emph{expected improvement acquisition function}: we will return to this in \cref{sec:beyond}.

The example from \cref{fig:pandora_example} we have just worked through demonstrates that \emph{one-step lookahead is suboptimal}.
In particular, one can check that P1 is the one-step lookahead policy, and its expected value is $0.1$ less than that of~P2.
This might not seem so bad, but there are other Pandora's box instances where the performance gap between one-step lookahead and the optimal policy can be arbitrarily large---a performance counterexample is given by \citet[Appendix~A.1]{singla_price_2018}.
With this background, we proceed to study the structure of P2.

\subsection{Defining the Gittins index for Pandora's box}
\label{sec:pandora:gittins}

We have seen in \cref{sec:pandora:example} that one-step lookahead does not solve Pandora's box.
To find the optimal policy, we essentially had to resort to brute force in \cref{eq:pandora_example_brute_force_1, eq:pandora_example_brute_force_2}.
This will not work in larger instances with more boxes.
Can we still solve Pandora's box in such cases?

Remarkably---as first shown by \citet{weitzman_optimal_1979} for Pandora's box, and \citet{gittins_bandit_1979} in a general abstract setting---the optimal policy, which is called the \emph{Gittins index policy} or, more concisely, the \emph{Gittins policy}, is nearly as simple as one-step lookahead.
Specifically, both the Gittins policy and one-step lookahead are \emph{index policies}: they work by computing a numeric rating for each box, called the box's \emph{index}, then opening the box with the best index.
Such policies are also known in the Bayesian optimization literature as \emph{acquisition functions}: we return to this connection in \Cref{sec:beyond}.
We consider the following index policies:
\*[a.] Under one-step lookahead, a box's index is its expected improvement over the current best value.
\* Under the Gittins policy, a box's index is a quantity called, appropriately, its \emph{Gittins index}.
\*/
Below, we give a quick definition of the Gittins index of a (closed) box and discuss its relationship to expected improvement.
Later on, we discuss in more depth where this definition comes from, and \emph{why} it is natural.

We first briefly review one-step lookahead.
Suppose box~$i$ is closed, and suppose the current best value is~$v_{\max}$.
One-step lookahead sets box~$i$'s index to $\EI_i(v_{\max})$, defined in~\cref{eq:pandora_ei}.
Roughly speaking, this index answers the question: \emph{how valuable is it to open box~$i$?}

The Gittins index comes from a related but different question.
\emph{Ignoring the actual value of $v_{\max}$}, we ask: \emph{hypothetically, if we had $v_{\max} = \alt$, how large would $\alt$ need to be to rule out opening box~$i$?}
If box's $i$'s expected improvement over~$\alt$ were negative, say if $\EI_i(\alt) < 0$, we could rule out opening box~$i$: simply stopping with reward~$\alt$ would be a better action.
So, define the \emph{Gittins index of box~$i$}, denoted~$G_i$, to be the solution to the root-finding problem
\[
  \label{eq:pandora_gittins}
  \EI_i(G_i) = 0
\]
or, equivalently, $\E{\max(v_i - G_i, 0)} = c_i$.
To see that the Gittins index is well defined, meaning there is a unique solution $G_i$ in \cref{eq:pandora_gittins}, note that $\EI_i(\alt)$ is convex (and hence, since its domain is the real line, continuous), decreasing as a function of~$\alt$, and~satisfies
\[
\lim_{\alt \to \infty} \EI_i(\alt) = -c_i < 0 < \infty = \lim_{\alt \to -\infty} \EI_i(\alt)
.
\]
Note also that a higher Gittins index corresponds to a more desirable box.
For instance, increasing $c_i$ decreases~$G_i$.\footnote{Here and throughout, \emph{increasing} and \emph{decreasing} are meant in their non-strict forms by default.}

Having defined the Gittins index, we can define how the Gittins policy behaves:
\* Let $v_{\max}$ be the maximum realized value among opened boxes.
\* If $\max_i G_i < v_{\max}$, meaning if every box's Gittins index is worse (less) than $v_{\max}$, then stop and select the best open box.
\* Otherwise, open $\argmax_i G_i$, the box of best (greatest) Gittins index.
\*/
In fact, if we extend the Gittins index to also cover open boxes, by letting the Gittins index of an open box be its revealed reward value, then---up to a choice of how to resolve ties---the Gittins policy reduces to one rule:
\*[{(\ensuremath{\star})}] \emph{Always take the action of greatest Gittins index.}
\*/
Returning to the example boxes in \cref{fig:pandora_example}, their Gittins indices are
\[
  \EI_1(\alt) &= \tfrac{1}{2} \cdot \max(14 - \alt, 0) - 1
  \qquad\Rightarrow\qquad  g_1 = 12
  \\
  \EI_2(\alt) &= \tfrac{1}{5} \cdot \max(18 - \alt, 0) - 1
  \qquad\Rightarrow\qquad g_2 = 13
  .
\]
Box~$2$ thus has better Gittins index than either box~$1$ or the open box ($v_{\max} = V_3 = 10$), so the Gittins policy would open box~$2$---the optimal action we saw in \cref{sec:pandora:example}!

\subsection{Extensions and limitations: what else can the Gittins index do?}
\label{sec:pandora:limitations}

At a high level, the Pandora's box problem is a model of \emph{search with costly information acquisition}.
Our goal is to in find a good value, but without paying too much opening costs.
In general, search problems like this can, at first, appear to be completely different than Pandora's box---indeed, the seminal work of \citet{gittins_dynamic_1974} was motivated by discounted Bayesian multi-armed bandits (\cref{sec:examples:bayesian_bandits}).
Even direct generalizations can have many features that go beyond the classical Pandora's box:
\*[1.] \emph{Multiple stages of inspection}. In Pandora's box, the reward is revealed immediately. In general, it might be revealed gradually over time (\cref{sec:examples:two-stage_pandora}).
\* \emph{Searching for multiple values}. In Pandora's box, our reward comes from just one box. In other settings, it may depend on the state of multiple boxes (\cref{sec:examples:finish_multiple}).
\* \emph{Dynamic sets of options}. In Pandora's box, the set of boxes is fixed. In general, new actions might become possible over time, such as when scheduling a set of randomly arriving jobs in a queueing system (\cref{sec:examples:queue}), or more generally (\cref{sec:examples:branching}).
\* \emph{Correlations between values}. In Pandora's box, the rewards of different boxes are assumed independent. In general, they can be correlated. This is common in Bayesian optimization, where rewards are modeled using Gaussian processes (\cref{sec:beyond:bayesopt}).
\* \emph{Optional inspection}. In Pandora's box, a box must be opened before being selected. In general, it might be possible to select a box and receive its reward without first opening it (\cref{sec:beyond:optional_inspection}).
\* \emph{Metrics beyond expected value}. In Pandora's box, the objective is to maximize expected net reward. In general, one might hope to optimize metrics beyond expected value. For instance, when scheduling in a queueing system, it is often more important to prevent very long delays than to reduce the average delay (\cref{sec:beyond:tail}).
\*/
In which of these situations can the Gittins index be defined?
In what cases is the resulting Gittins policy still optimal?
What if we have several of these aspects simultaneously, such as multiple stages of inspection, with parts that can be skipped?

A principal aim of the rest of this tutorial is to shed light on these questions.
To do so, we first need to be able to precisely describe the potential features listed above.
To that end, over the next few sections, we present a unifying decision-making framework which includes Pandora's box---as well as well as related problems from seemingly different settings such as optimal queueing---into a common language.

At this level of generality, the Gittins index can be defined, and the resulting Gittins policy is optimal.
We then generalize our framework further to capture the more difficult of the advanced variants presented above: in most such situations, the Gittins index can still be defined, but can only be expected to yield a strong policy---in the spirit, of, say, upper confidence bounds or information-theoretic decision rules---rather than an outright optimal one.
We conclude the tutorial by surveying applications where the Gittins policy, despite being suboptimal, has excellent theoretical or empirical performance.

\section{General formulation of the Gittins index}
\label{sec:general}

\newcommand{\figPandoraMarkovChain}{%
  \begin{figure}[t]
    \begin{tikzpicture}
      \drawBoxClosed{-6.5, 0.5}{$\boxclosed$}

      \coordinate (A) at (-2.5, 0.5);
      \draw (-5.5, 0.5) -- (A)
        node[midway, above] {reward $-c$}
        node[midway, below] {next state $\sim p$}
        -- +(0.01, 0);

      \node at (-1.625, +3) {$\vdots$};
      \draw[<-] (-1.25, +2) -- +(-0.25, 0) to[out=180, in=0] (A);
      \draw[<-] (-1.25, +1) -- +(-0.25, 0) to[out=180, in=0] (A);
      \draw[<-] (-1.25, 0) -- +(-0.25, 0) to[out=180, in=0] (A);
      \draw[<-] (-1.25, -1) -- +(-0.25, 0) to[out=180, in=0] (A);
      \node at (-1.625, -1.75) {$\vdots$};

      \node at (0, +3) {$\vdots$};
      \drawBoxOpen{0, +2}{$3.2$}
      \drawBoxOpen{0, +1}{$3.1$}
      \drawBoxOpen{0, 0}{$3.0$}
      \drawBoxOpen{0, -1}{$2.9$}
      \node at (0, -1.75) {$\vdots$};

      \coordinate (C) at (5, 0.5);
      \draw[<-] (5.25, 0.5) -- (C) -- +(-0.01, 0);

      \node at (1.625, +3) {$\vdots$};
      \draw (1.25, +2) -- +(2.25, 0) node[midway, above] {reward $3.2$} to[out=0, in=180] (C);
      \draw (1.25, +1) -- +(2.25, 0) node[midway, above] {reward $3.1$} to[out=0, in=180] (C);
      \draw (1.25, 0) -- +(2.25, 0) node[midway, above] {reward $3.0$} to[out=0, in=180] (C);
      \draw (1.25, -1) -- +(2.25, 0) node[midway, above] {reward $2.9$} to[out=0, in=180] (C);
      \node at (1.625, -1.75) {$\vdots$};

      \drawBoxSelected{6.5, 0.5}{$\checkmark$}
    \end{tikzpicture}
    \caption{
      Illustration of the Markov chain for a box with opening cost~$c$ and reward distribution~$p$.
      The states are \emph{closed}, denoted~$\boxclosed$; \emph{opened with reward $v \in \bbR$}, denoted~$v$; and \emph{selected}, denoted~$\checkmark$.
    }
    \label{fig:pandora_markov_chain}
  \end{figure}
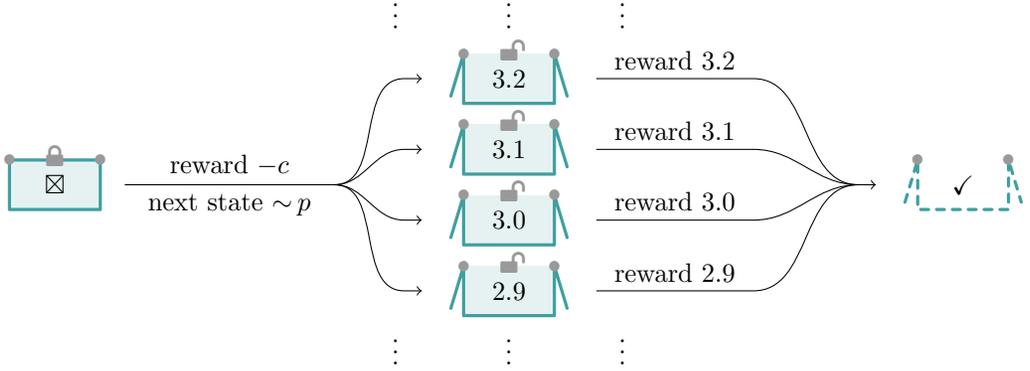%
}

Having seen the Pandora's box problem, its formulation as a Markov decision process, and its solution, one can ask: \emph{is there a general theory this solution is an example of?}
We now present such a theory, focusing on a formulation that generalizes well beyond Pandora's box.

The key idea is to conceptualize each individual Pandora's box as a \emph{transient Markov chain} with one absorbing state, with the collection of all boxes represented as tuples whose $i$th element is the Markov chain state corresponding to box $i$. \cref{fig:pandora_markov_chain} provides an illustration.
\figPandoraMarkovChain
An action $i$ then amounts to advancing the $i$th Markov chain forward by one step, and collecting whatever rewards arise as a result---where costs are represented as negative rewards.
The agent must select which Markov chain to advance at each time point.

This formulation enables one to handle many situations which, at first, appear to have little to do with Pandora's box.
This includes Bayesian variants of the multi-armed bandit problem \cite{gittins_dynamic_1974, gittins_bandit_1979} and mean delay minimization in single-server queues \cite{sevcik_use_1971, sevcik_scheduling_1974, olivier_kostenminimale_1972, klimov_time-sharing_1974, klimov_time-sharing_1978}, both of which predate the Pandora's box of \citet{weitzman_optimal_1979}, and for which Gittins indices were discovered independently.
We now study a framework that enables one to see how optimal policies arise in all of these setups.

\subsection{The Markov chain selection decision problem}
\label{sec:general:mcs}

We begin by formulating our general decision problem, which we call \emph{Markov chain selection (MCS)}, as a Markov decision process.

\begin{definition}[Markov chain]
  \label{def:markov_chain}
  A \emph{Markov chain\footnote{Note that our definition of a Markov chain (without rewards) is is equivalent to the usual random-variable-theoretic formulation from probability theory. Here, we opt to work with transition kernels, rather than collections of random variables, because this will make it notationally cleaner to track relationships between various different Markov decision processes that arise in our context.} with rewards} is a tuple $(S, \partial S, p, r)$ consisting of:
  \*[1.] The \emph{state space} $S$.
  \* A subset of \emph{terminal states} $\partial S \subseteq S$, which may be empty.
  \* A \emph{transition kernel} $p : S \to \calP(S)$.
  \* A \emph{reward function} $r : S \to \bbR$, which may take negative values.
  \*/
  We require that the terminal states be absorbing with zero reward, namely $p(s)$ deterministically maps $s \in \partial S$ to itself, and $r(s) = 0$ for all $s \in \partial S$.
\end{definition}

To ease terminology, we often omit \emph{with rewards}.
Note that when specifying a Markov chain, it suffices to define the part of the transition kernel which describes transitions out of non-terminal states, and similarly it suffices to define rewards for non-terminal states.

\begin{definition}[Markov chain selection]
\label{def:mcs}
Define the \emph{Markov chain selection (MCS)} problem with Markov chains $(S_i, \partial S_i, p_i, r_i)$, for $i \in \{1, \dots, n\}$, to be the MDP given as follows:
\*[1.] State space: let the state space be
\[
S_\mcs = \{(s_1, \dots, s_n) : s_i \in S_i \text{ for all } i\}
\]
together with an initial state whose components are all non-terminal.
\* Terminal states: let the terminal state set be
\[
  \partial S_\mcs = \{(s_1, \dots, s_n) \in S_\mcs : s_i \in \partial S_i \text{ for some } i\}
  .
\]
\* Action space: let $A_\mcs = \{1, \dots, n\}$.
\* Reward function: let
\[
r_\mcs(s,a) = r_a(s_a)
.
\]
\* Transition kernel: given a state $(s_1, \dots, s_n)$ and action $a$, we transition the MDP into a new state by replacing $s_a$ with $s_a' \sim p_a(s_a)$ according to that respective Markov chain's transition kernel, leaving other states unchanged.
\* Discount factor: let $\dscnt\in(0,1]$.
\*/
\end{definition}

To understand this definition, note that Pandora's box is a special case where each Markov chain, illustrated \cref{fig:pandora_markov_chain}, is as follows.
For box~$i$, the state space is $S_i = \{\boxclosed\}\cup\bbR\cup\{\checkmark\}$ and transition kernel given by $p_i(\boxclosed) = p_i$ and $p_i(v_i) = p_i(\checkmark) = \delta_\checkmark$, where $\delta_\checkmark$ is the Dirac measure at the symbol $\checkmark$.
In this Markov chain, $\checkmark$ is the unique absorbing state, which is also terminal, and all other states are transient.
Here, we take $\gamma=1$.

In general, choosing an action thus corresponds to choosing which Markov chain to transition---an abstract generalization of choosing which box to open.
Our formulation works with time-homogeneous Markov chains: this is without loss of generality, as non-time-homogeneous chains can be handled by adjoining time to their state space, transforming them into time-homogeneous ones.
To ensure well-definedness, we make the following assumption on each individual Markov chain.

\begin{assumption}
\label{asm:discount}
At least one of the following two conditions holds:
\*[subenv] From any initial state, the Markov chain with transition kernel $p$ reaches a terminal state in finite time with probability one, and the sum of absolute values of rewards of all transitions until termination has finite expectation.
\label{asm:discount:non_discounted}
\* We have $\dscnt<1$, and $r$ is uniformly bounded in absolute value.
\label{asm:discount:discounted}
\*/
\end{assumption}

Our arguments will proceed by showing that \cref{sub@asm:discount:discounted} reduces to \cref{sub@asm:discount:non_discounted}, then studying that case.
It is also possible to work in slightly greater generality: we adopt the formulation here to minimize technicalities while keeping results sufficiently general.

Compared to the situation of Pandora's box, it is, at first, even less obvious whether anything can be said about the defined MDP's optimal policy---particularly given that the MDP is much less concrete than before, and its abstract nature could potentially give rise to rather different looking examples.

The insight of \citet{gittins_bandit_1979}---which is particularly remarkable given it was first discovered in essentially the same generality we consider here---is that this class of MDPs can be solved in much the same manner as we presented Pandora's box in \cref{sec:pandora}.
Namely, the idea will be to look for a way to compare Markov chains with each other, just as before we looked for a way to compare Pandora's boxes with each other.
Specifically, we seek a way compare Markov chains, which are stochastic, with real numbers, which are not.

\subsubsection{A note on the name \emph{Markov chain selection}}
\label{sec:general:mcs:name}

We conclude with a note on terminology.
What we call MCS is, roughly speaking, usually called the \emph{Markovian multi-armed bandit problem} in the Gittins index literature.
We introduce the new name \emph{Markov chain selection} for two reasons.
First, the Markovian multi-armed bandit is typically defined slightly more restrictively, namely ruling out undiscounted settings and terminal states \citep{gittins_multi-armed_2011}, and we want to emphasize that we do not impose these restrictions.

Second, we wish to slightly \emph{de-emphasize} the link between Gittins indices and multi-armed bandits.
The broader multi-armed bandit literature is vast \citep{lattimore_bandit_2020, slivkins_introduction_2019}, and in the context of this vast literature, the Gittins index might only seem useful for solving one corner case, namely discounted Bayesian bandits (\cref{sec:examples:bayesian_bandits}).
On the other hand, in our opinion, \emph{many applications of the Gittins index do not superficially resemble bandit problems}.
The feature that unifies Gittins index applications---Pandora's box, bandits, queue scheduling, and more---is repeatedly choosing which of multiple independent Markov chains to advance.
Our hope is that the \emph{Markov chain selection} name directly evokes this unifying feature.

\subsection{Defining the Gittins index via the local MDP}

\newcommand{\figureLocalMdpValue}{%
  \begin{figure}[t]
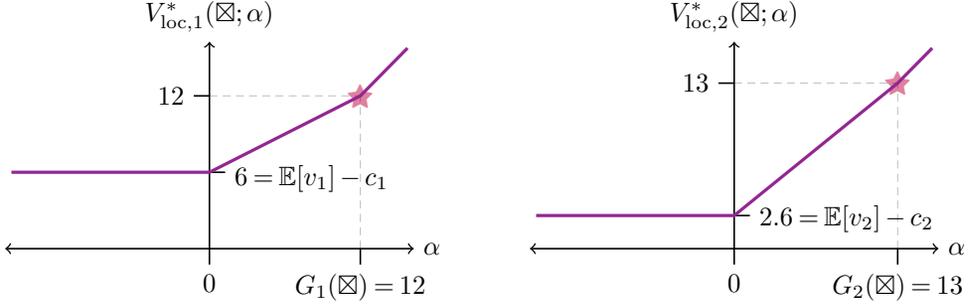

    \hfill%
    \begin{subfigure}{0.475\linewidth}
      \pictureLocalMdpValuePandora{6}{12}{1}
      \caption{Local MDP value for box~$1$ from \cref{fig:pandora_example}.}
    \end{subfigure}%
    \hfill%
    \begin{subfigure}{0.475\linewidth}
      \pictureLocalMdpValuePandora{2.6}{13}{2}
      \caption{Local MDP value for box~$2$ from \cref{fig:pandora_example}.}
    \end{subfigure}%
    \hfill%
    \caption{
      Optimal value of the $(\boxclosed, \alpha)$-local MDP for the two closed boxes in \cref{fig:pandora_example}.
      We include the box identifier in subscripts throughout for disambiguation.
      For sufficiently small~$\alpha$, namely $\alpha \leq 0$, the optimal policy opens the box and takes its reward, meaning it plays $\go$ twice, yielding value $\E{v_i} - c_i$.
      For intermediate values of~$\alpha$, the optimal policy opens the box but might then take the alternative, meaning it plays $\go$ at least once.
      The value's slope in this regime is the probability of eventually taking the alternative, namely (a)~$\frac{1}{2}$ or (b)~$\frac{4}{5}$.
      For sufficiently large~$\alpha$, the optimal policy takes the alternative without opening the box, meaning it plays $\st$.
      The value of~$\alpha$ at the boundary between the latter two regimes---that is. the unique value at which both $\go$ and $\st$ are optimal---is the Gittins index~$G_i(\boxclosed)$.
    }
    \label{fig:local_mdp_value}
  \end{figure}%
}

Mirroring the approach used in Pandora's box---namely, defining the Gittins index of a single closed box---we now define the Gittins index of a general Markov chain equipped with a general reward function.
We do so using the following notion.

\begin{definition}[Local MDP]
\label{def:local_mdp}
Let $(S,\partial S, p, r)$ be a Markov chain satisfying \cref{asm:discount}.
For every \emph{alternative option} $\alpha \in \bbR$ and \emph{initial state} $s\in S$, define a Markov decision process, called the \emph{$(s, \alpha)$-local MDP}, as follows:
\*[1.] State space: let $S_\loc = S \cup \{\checkmark\}$,\footnote{Throughout this work, we adopt the convention that unions with symbols such as $\checkmark$ are disjoint unions. That is, it is understood that the union $S \cup \{\checkmark\}$ uses a symbol $\checkmark\notin S$ which is not part of the original states.} with initial state~$s$.
\* Terminal state: $\partial S_\loc = \{\checkmark\}$.
\* Action space: let $A_\loc = \{\go,\st\}$, called \emph{go} and \emph{stop}, respectively.
\* Reward function: for $s \in S$, let $r_\loc(s,\go) = r(s)$, $r_\loc(s,\st) = \alpha$, and $r_\loc(\checkmark,\go) = r_\loc(\checkmark,\st) = 0$.
\* Transition kernel: if $s \in S$ and $a=\go$, then let $s'\sim p(s)$, otherwise if $s=\checkmark$ or $a=\st$ let $s' = \checkmark$.
\* Discount factor: let $\dscnt\in(0,1]$.
\*/
\end{definition}

The intuition behind the local MDP is that it simplifies a \emph{global} MCS instance down to the \emph{local} perspective of one Markov chain.
Specifically, it captures the tradeoff between advancing the Markov chain in state~$s$ vs. taking some other action.
While MCS has are many other actions to choose from, the local MDP simplifies the tradeoff by providing just one other action $\st$ with a very clear value~$\alpha$.
This generalizes the idea from \cref{sec:pandora:gittins} of comparing a closed Pandora's box with an open box---that is, comparing whether to transition the respective Markov chain one state forward, or simply take the value $\alpha$ from an alternative~open~box.

To make this precise, let $V_\loc^*(s; \alpha)$ be the value function of the local MDP, which is well defined by \cref{asm:discount}.
In the \emph{undiscounted} setting, we can write
\[
    \label{eq:v_loc_sup}
    V_\loc^*(s; \alpha) = \sup_{\pi : S_\loc \to A_\loc} (E_\pi(s) + \alpha P_\pi(s))
\]
where $E_\pi(s)$ is the expected total reward $\pi$ receives while playing $\go$ when starting from~$s$, and $P_\pi(s)$ is the probability $\pi$ eventually plays $\st$ when starting from~$s$.
In the \emph{discounted} setting, we can still write \cref{eq:v_loc_sup}, but $E_\pi(s)$ is discounted reward, and $P_\pi(s) = \E{\dscnt^{t_\pi}}$, where $t_\pi$ is either the time when $\pi$ plays~$\st$, or $\infty$ if $\st$ is never played.
\cref{fig:local_mdp_value} gives an illustration of the local MDP's value function for the closed boxes from \cref{fig:pandora_example}.
\figureLocalMdpValue

The general definition of the Gittins index of a state~$s$ similarly generalizes the Pandora's box Gittins index definition from \cref{sec:pandora:gittins}.
The definition follows from two key observations about the local MDP with initial state~$s$ and varying alternative:
\*[1.] If playing $\st$ is optimal when the alternative is~$\alpha$, then $\st$ is still optimal for any \emph{better} alternative $\alpha' > \alpha$.
\* If playing $\go$ is optimal when the alternative is~$\alpha$, then $\go$ is still optimal for any \emph{worse} alternative $\alpha' < \alpha$.
\*/
One can prove this using the fact that $\st$ is optimal if and only if $V_\loc^*(s; \alpha) = \alpha$ and the following properties of $V_\loc^*$.

\begin{lemma}
\label{lem:v_loc}
For any state $s$ of any Markov chain satisfying \cref{asm:discount}, the function $\alpha \mapsto V_\loc^*(s; \alpha)$ has the following properties:
\*[subenv] It is convex and non-decreasing.
\* It has left and right derivatives taking values in $[0, 1]$.
\* It is bounded below by $V_\loc^*(s; \alpha) \geq \alpha$.
\*/
\end{lemma}

\begin{proof}
We argue as follows:
  \*[(a)] By \cref{eq:v_loc_sup}, $\alpha \mapsto V_\loc^*(s; \alpha)$ is a supremum of convex (namely, affine) and non-decreasing functions, so it is also convex non-decreasing.
  \* This follows from \cref{eq:v_loc_sup}, the convexity of $\alpha \mapsto V_\loc^*(s; \alpha)$, the fact that $P_\pi(s) \in [0, 1]$, and a standard envelope theorem \citep[Theorem~1]{milgrom_envelope_2002}.
  See \citet[Appendix~B.6]{xie_cost-aware_2024} for additional discussion about envelope theorems in this setting.
  \* Playing $\st$ immediately yields value~$\alpha$, and the optimal policy does at least as well.
  \*/
\end{proof}

Using these properties, one can show that there is a unique alternative~$\alpha$ such that \emph{both} $\st$ and $\go$ are optimal, namely the minimum value of~$\alpha$ such that $\st$ is optimal.
We define the Gittins index $G(s)$ of~$s$ to be this transition point.
\Cref{fig:local_mdp_value} illustrates an example.

\begin{definition}[Gittins index]
\label{def:gittins_general}
Let $(S, \partial S, p, r)$ be a Markov chain satisfying \cref{asm:discount}.
The \emph{Gittins index function} for the Markov chain, denoted $G : S \to \bbR\cup\{\infty\}$, maps each state $s\in S$ to either the unique number $g\in\bbR$ such that both $\go$ and $\st$ are optimal actions for the $(g, s)$-local MDP at its initial state, or to $\infty$ if no such number exists.
\footnote{
  Under \cref{asm:discount}, it is always possible to pass to the setting where $G(s)$ is necessarily finite: one can replace the given MCS instance with a modified MCS instance over a slightly smaller state space in a manner that preserves value functions and policies.
  See \cref{asm:simplify:no_free_states} and the following discussion for details.
}
\end{definition}

We call $G(s)$ the \emph{Gittins index of state~$s$}, and there are many equivalent ways to formulate it, for instance using stopping times.
In particular, one can define the Gittins index to be the supremum of alternative values for which $\go$ is strictly optimal, or the minimum value of~$g$ such that playing $\st$---which we note results in value~$g$---is optimal:
\[
  \label{eq:gittins_general}
  G(s)
  = \sup\curlgp{g \in \bbR : V_\loc^*(s; g) > g}
  = \inf\curlgp{g \in \bbR : V_\loc^*(s; g) = g}
\]
with the convention that infima and suprema of empty sets are $\infty$.
One can similarly view $G(s)$ as the solution to a root finding problem, as we did for Pandora's box in \cref{eq:pandora_gittins}:
letting $V_\loc^\goName(s; g)$ be the optimal value achievable when playing $\go$ at least once, we have
\[
  \label{eq:gittins_root}
  G(s) = V_\loc^\goName(s; G(s))
  .
\]
There are other ways to express $G(s)$, particularly in terms of optimization problems over stopping policies, or equivalently, stopping sets.
We refer the interested reader to prior expositions of the Gittins index for details \citep{gittins_multi-armed_2011, zhao_multi-armed_2020, balakrishnan_multi-armed_2014, glazebrook_stochastic_2014}.

When needed for disambiguation between multiple Markov chains $(S_i, \partial S_i, p_i, r_i)$, we add a subscript~$i$ as needed, for instance $G_i$ and $V_{\loc, i}^*$, as was used in \cref{fig:local_mdp_value}.

\subsection{Optimality of the Gittins policy}

We now state the key optimality theorem for \cref{def:gittins_general}, deferring its proof to \cref{sec:proof}.

\begin{theorem}
\label{thm:gittins_optimal}
Consider an instance of MCS (\cref{def:mcs}) with all Markov chains satisfying \cref{asm:discount}.
A policy for MCS is optimal if and only if it always selects an action of maximal Gittins index---meaning, if when in state $(s_1, \dots, s_n)$, it selects
\[
a \in \argmax_{i \in \{1, \dots, n\}} G_i(s_i)
.
\]
\end{theorem}

In general, there may be several such policies, parameterized by an appropriately defined tie-breaking rule which chooses a maximizer in the event it is non-unique.
We will implicitly assume that a tie-breaking rule has been chosen, and will refer to any policy satisfying the condition in \cref{thm:gittins_optimal} as \emph{the Gittins policy}.
We can summarize \cref{thm:gittins_optimal} as follows:
\*[{(\ensuremath{\star})}]\emph{Always choose the Markov chain of greatest Gittins index.}
\*/
One can view this not just as a policy for MCS, but as a general design principle.
If we can compare an action to a deterministic alternative in the style of the local MDP, then we can usually define the action's Gittins index just as in \cref{def:gittins_general}.
For example, in \cref{sec:beyond:bayesopt}, we explain how this design principle applies to Bayesian optimization, a problem which may at first appear dissimilar to MCS.

In the situation formalized here---and a small set of generalizations, some of which we discuss in further detail in \cref{sec:examples}---this decision-making principle is outright optimal.
However, the specific statement in \cref{thm:gittins_optimal} is rather fragile.
It is not uncommon, in more general situations, for optimality to fail, and do so in a manner that suggests the proof technique breaks down completely.

In such situations, it be very tempting to conclude that this is because Gittins indices are not a good approach.
However, there are cases where the definition continues to make intuitive sense, and one can either (a)~show various notions of near-optimality, such as regret \citep{lattimore_regret_2016, farias_optimistic_2022} or approximation ratio \citep{scully_local_2024, chawla_combinatorial_2024, gergatsouli_weitzmans_2023} bounds, or (b)~show strong empirical performance \cite{xie_cost-aware_2024}.
In such cases, it can be more insightful to instead interpret the lack of an optimality theorem as a \emph{statement about the problem's richness}, rather than of Gittins-index-style stochastic-to-deterministic comparisons being the wrong decision-making approach.

Understanding where the Gittins policy is strong, even if not outright optimal, remains an active research area.
In \cref{sec:beyond}, we cover some examples where the Gittins policy is suboptimal but nevertheless has strong theoretical or empirical performance.

\subsection{Computation}
\label{sec:general:computation}

We now discuss computational properties of the Gittins index.
This problem is well studied for finite-state Markov chains: \citet{balakrishnan_multi-armed_2014} give a survey of the topic.
The current state-of-the-art algorithm is that of \citet{gast_testing_2023}, which runs in sub-cubic time.
We note also that even faster algorithms are possible under additional structure: see for instance \citet[Section~4.2 and Appendix~B]{scully_optimal_2018}.

For infinite-state Markov chains, computation remains a significant challenge \citep{kelly_multi-armed_1981, kim_robust_2016, farias_optimistic_2022, gittins_multi-armed_2011}.
Continuous state spaces have received comparatively little attention beyond Pandora's box, and we view computing the Gittins index of continuous-state Markov chains to be a significant open problem: we return to this in \cref{sec:conclusion:open_problems}.
We now discuss the key challenges for~doing~so.

As expressed by \cref{eq:gittins_root}, computing the Gittins index $G(s)$ of a given Markov chain state $s$ requires one to solve a root-finding problem defined in terms of the local MDP's value function $V_\loc^*$.
As such, one can expect computation of $G(s)$ in a given concrete setting to potentially involve elements of dynamic programming, together with root-finding algorithms.
This gives rise to two~challenges:
\*[1.] For continuous-state Markov chains, one must usually perform dynamic programming approximately \cite{bertsekas_dynamic_2012}.
\* The local MDP must be solved for enough different values of the parameter~$\alpha$ to determine the Gittins index.
\*/
In some classical problems---for instance, Gaussian Pandora's box---one can compute $V_\loc^*$ analytically.
In such cases, owing to monotonicity of $\alpha \mapsto V_\loc^*(s; \alpha)$, the Gittins index can be computed efficiently using bisection search.

Even when $V_\loc^*$ cannot be computed analytically, we suspect one can do better than simply apply off-the-shelf approximate dynamic programming algorithms.
This is because the local MDP possesses two properties which generic MDPs do not: its action space $\{\st,\go\}$ is very small, and the value of $\st$ is always $g$.
At present, methods that leverage these properties have largely yet to be developed.
The key challenge is in effectively handling the state space $S$ in situations where it is sufficiently high-dimensional to render discretization unviable.

Finally, there are extensions of MCS where there are infinitely many Markov chains, possibly uncountably many.
This introduces another obstacle: finding the chain of maximum Gittins index.
In general, this needs to be performed using gradient-based optimization.
We discuss how to do so in the context of Bayesian optimization in \cref{sec:beyond:bayesopt:computation}.

\section{Examples: optimal policies}
\label{sec:examples}

We now work through a number of examples.
We begin with a multi-stage Pandora's box, as a simple illustration of how the developed framework allows us to handle a slight generalization of our initial illustrative setup (\cref{sec:examples:two-stage_pandora}).
We then follow up with a discounted Bayesian bandit, illustrating how discounted chains without terminal states are handled (\cref{sec:examples:bayesian_bandits}).
Finally, we discuss three mild generalizations of the developed framework: (a) a variant which models Markovian search for multiple items, which therefore involves termination of more than one Markov chain (\cref{sec:examples:finish_multiple}); (b) a queueing example with a seemingly different objective, but which is handled similarly to the aforementioned Markovian search (\cref{sec:examples:queue}); and (c) a branching bandit, where the action space at each time point varies stochastically in a manner that can depend on the chosen actions (\cref{sec:examples:branching}).
In all cases, an optimality result in the spirit of \cref{thm:gittins_optimal} continues to hold.

\subsection{Two-stage Pandora's box}
\label{sec:examples:two-stage_pandora}

Consider the following variant of Pandora's box where there are two stages of inspection for a box.
\*[1.] The first stage reveals some partial information about the box's contents, which we call its \emph{label}.
This inspection incurs some cost.
\* The second stage opens the box and reveals its reward, just like in ordinary Pandora's box.
This incurs additional cost.
\* Finally, we can select a box and receive its reward once it is fully open.
\*/
To model a two-stage box as a Markov chain, we use much the same approach as for ordinary Pandora's box in \cref{fig:pandora_markov_chain}, but with additional states representing the label.
Specifically, letting $L$ be the set of labels, assumed disjoint from the remaining states, the state space is
\[
  S = \{\boxclosed\} \cup L \cup \bbR \cup \{\checkmark\}
  .
\]
The transition kernel $p : S \to \calP(S)$ described above then has the following form.
\*[1.] For the first stage, we always transition from $\boxclosed$ to some label $\ell \in L$. We thus write $p(\boxclosed) \in \calP(L)$ as the distribution over labels.
\* For the second stage, when in state $\ell \in L$, we transition to a state $v \in \bbR$ by sampling $v \sim p(\ell)$.
Here, $p(\ell) \in \calP(\bbR)$ is the reward distribution of a box given that its label is~$\ell$.
\* Finally, from any fully inspected state $v \in \bbR$, we always transition to the selected state~$\checkmark$, so $p(v) = \delta_\checkmark$.
\*/
This encodes how information is revealed as part of the inspection process.
Similarly, the reward function $r : S \to \bbR$ encodes the inspection costs.
\*[1.] If the cost of the first inspection stage is~$c(\boxclosed)$, then $r(\boxclosed) = -c(\boxclosed)$.
\* If, given a label $\ell$, the cost of the second stage is $c(\ell)$, then $r(\ell) = -c(\ell)$.
\* Finally, for any fully inspected state $v \in \bbR$, the reward is simply $r(v) = v$.
\*/

What does the Gittins index look like for this two-stage box?
Following \cref{thm:gittins_optimal}, to compute the Gittins index $G(s)$ of a state~$s$, we need to understand how the local MDP's value function $V_\loc^*(s; \alpha)$ depends on $\alpha$.
Specifically, the point where $\st$ and $\go$ are co-optimal is by definition also the smallest value of $\alpha$ for which $V_\loc^*(s; \alpha) = \alpha$.

To obtain the value function, we apply dynamic programming to the local MDP's, which is tractable because there are only two possible actions.
Working backwards from the terminal state $\checkmark$, we find:
\*[][3.] For fully inspected states $v \in \bbR$, we clearly have
\[
  V_\loc^*(v; \alpha) = \max(v, \alpha)
  .
\]
So, as in the one-stage Pandora's box, co-optimality implies $G(v) = v$.
\*[2.] For partially inspected states $\ell \in L$, we essentially have the same situation as for ordinary Pandora's box:
\[
  V_\loc^*(\ell; \alpha)
  &
  = \max\gp[\big]{\E_{v \sim p(\ell)}{V_\loc^*(v; \alpha)} - c(\ell), \alpha}
  \\ &
  = \max\gp[\big]{\E_{v \sim p(\ell)}{\max(v, \alpha)} - c(\ell), \alpha}
  .
\]
So computing $G(\ell)$ amounts to the same root-finding problem as in~\cref{eq:pandora_gittins}, namely
\[
  \E_{v \sim p(\ell)}{\max(v - G(\ell), 0)} - c(\ell) = 0
  .
\]
\*[1.] Finally, for the initial uninspected state~$\boxclosed$, we have
\[
V_\loc^*(\boxclosed; \alpha)
= \max\gp[\big]{\E_{\ell \sim p(\boxclosed)}{V_\loc^*(\ell; \alpha)}  - c(\boxclosed), \alpha}
\]
which yields a second kind of root-finding problem for $G(\boxclosed)$, namely
\[
\E_{\ell \sim p(\boxclosed)}{V_\loc^*(\ell; \alpha) - \alpha} - c(\boxclosed) = 0
.
\]
\*/

The Gittins index for this problem can therefore be expressed as a solution to a sequence of root-finding problems.
One can generalize this to more than two stages while maintaining a similar structure.
In general, following \cref{sec:general:computation}, one can expect that computing $G$ will require numerical methods---albeit ones which involve each local MDP individually, rather than the full Markov chain selection MDP.

\subsection{Bayesian Bernoulli bandits with discounting}
\label{sec:examples:bayesian_bandits}

In the \emph{Bernoulli multi-armed bandit problem}, an agent is presented with $n$ coins, where each coin~$i$ has unknown heads probability~$x_i$.
Our goal, roughly, is to repeatedly flip coins in a way that maximizes the expected (discounted) number of heads.
In the \emph{Bayesian} version of the problem, the agent has a prior distribution on~$x_i$.
This problem fits into the framework of \cref{def:mcs} with discount parameter $\dscnt < 1$: each Markov chain corresponds to one coin, advancing the Markov chain corresponds to flipping the coin: the Markov chain's state represents the current posterior distribution the agent has on~$x_i$.

We focus on the traditional beta-distributed prior $x_i \sim \dBeta(a_i, b_i)$ for some $a_i, b_i > 0$.
In this case, by standard properties of the beta distribution, we can express each coin~$i$ as a Markov chain as follows:
\*[1.] The state space is $S = (0, \infty)^2$, with an initial state is $(a_i, b_i)$.
\* The transition kernel is
\[
  p(a, b)
  = \begin{cases}
    (a + 1, b) & \t{w.p.} \frac{a}{a + b} \\
    (a, b + 1) & \t{w.p.} \frac{b}{a + b}
                 .
  \end{cases}
\]
\* The reward function is $r(a, b) = \frac{a}{a + b}$.
\*/

\Cref{def:mcs}, with Markov chains of the above, is therefore our Bayesian Bernoulli bandit of interest.
To compute the Gittins index, we need to solve the respective local MDP of \cref{def:local_mdp} for the above Markov chain.
However, because the Markov chain has a countably infinite state space and no terminal states, computing the Gittins index is nontrivial, though several practical approaches for approximating or bounding it are known: see \citet{kelly_multi-armed_1981, kim_robust_2016, farias_optimistic_2022, gittins_multi-armed_2011} for examples.

\subsection{Selecting multiple boxes, or finishing multiple Markov chains}
\label{sec:examples:finish_multiple}

One can view Pandora's box, and more generally MCS with terminating Markov chains, as a model of searching for a single item.
What if one is instead searching for multiple items?
It turns out that in the undiscounted setting---that is, with $\dscnt = 1$---the Gittins policy optimally solves this version of the problem, too.
Specifically, for any~$k$, the Gittins policy maximizes the total expected reward received up until the point the first $k$ Markov chains reach terminal states.
We now state this more formally.

\begin{definition}
  \label{def:mcs-k}
  The \emph{$k$-finish Markov chain selection (MCS-$k$)} problem for $k \leq n$ Markov chains is an MDP defined in the same way as \cref{def:mcs}, except the process does not end until $k$ Markov chains reach terminal states, meaning the terminal states are
  \[
    \partial S_\mcsk
    = \{(s_1, \dots, s_n) \in S_\mcs : s_i \in \partial S_i \text{ for } k \text{ distinct } i\}
    .
  \]
  Additionally, action~$i$ is disallowed if Markov chain~$i$ is in a terminal state, meaning this MDP has a state-dependent action space, which is defined to be
  \[
  A_\mcsk(s_1,\dots,s_n) = \{i : s_i \notin \partial S_i\}
  .
  \]
\end{definition}

The case $k=1$ reduces to \cref{def:mcs}.
The following result shows \cref{thm:gittins_optimal} generalizes to $k\geq 2$.
Moreover, the proof is very similar: see \cref{sec:proof:mcs-k}.

\begin{theorem}
\label{thm:gittins_optimal_mcs-k}
The Gittins policy is optimal for undiscounted MCS-$k$.
\end{theorem}

In fact, one can use the Gittins index to solve an even more general problem than MCS-$k$ involving \emph{combinatorial constraints}.
For example, consider the following problem introduced by \citet{singla_price_2018}, which we call \emph{spanning-tree Pandora's box}.
We are given a graph, where each edge of the graph is associated with a box with some opening cost and reward distribution.
The problem proceeds much like the $k$-finish variant of Pandora's box, except the set of boxes we take must contain no cycles.
One can view this as a max-weight spanning tree variant where we replace deterministic edge weights with Pandora's boxes.

The classical max-weight spanning tree problem is famously optimally solvable by greedy algorithms.
The simplest of these, arguably, is Kruskal's algorithm, which accumulates a spanning tree by repeatedly adding the edge of greatest weight that would not form a cycle.
\Citet{singla_price_2018} shows that essentially the same algorithm solves spanning-tree Pandora's box if one uses \emph{Gittins indices in place of deterministic edge weights}.
That is, at every time step, we take the action corresponding to the box of greatest Gittins index, provided that box's edge would not form a cycle with selected edges.

Moreover, there is nothing special about Pandora's box here, and indeed, \citet{gupta_markovian_2019} show that essentially the same algorithm solves the version of the problem where each edge has an arbitrary (terminating) Markov chain.
For example, one could use the two-stage Pandora's box variant of \cref{sec:examples:two-stage_pandora} to model max-weight spanning tree problems where each edge requires two rounds of inspection to reveal its value---perhaps an inexpensively acquired initial estimate followed by an expensively acquired precise valuation.
The resulting strategy of taking a classic combinatorial algorithm and plugging in Gittins indices in place of deterministic weights is potentially very general.

In what situations does this strategy work?
The answer, roughly speaking, is for \emph{greedy algorithms} of a certain form identified by \citet{singla_price_2018}.
For instance, one can optimally solve what we might call \emph{matroid-finish MCS}, because matroid packing problems---of which max-weight spanning tree is a special case---are solved by greedy algorithms.
Moreover, for situations where greedy algorithms only give approximately optimal solutions, \citet{singla_price_2018} (for Pandora's box) and \citet{gupta_markovian_2019} (for general Markov chains) show that the Gittins index version of the greedy algorithm achieves the same approximation ratio as the algorithm would achieve in the classical deterministic-weight setting.
Thus, one can view Gittins indices as providing a good definition for what \emph{greedy} should mean in stochastic contexts.

\subsection{Scheduling in queues}
\label{sec:examples:queue}

\newcommand{\figJobMarkovChainKnown}{%
  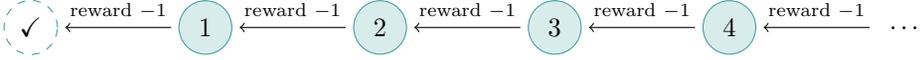
\begin{figure}[t]
    \begin{tikzpicture}
      \newcommand{\xscale}{2.5}
      \node[state selected] (A0) at (0 * \xscale, 0) {$\checkmark$};
      \node[state] (A1) at (1 * \xscale, 0) {$1$};
      \node[state] (A2) at (2 * \xscale, 0) {$2$};
      \node[state] (A3) at (3 * \xscale, 0) {$3$};
      \node[state] (A4) at (4 * \xscale, 0) {$4$};
      \node[state, draw=none, fill=none] (A5) at (5 * \xscale, 0) {$\dots$};
      \draw[<-] (A0) to node[midway, above] {\scriptsize reward $-1$} (A1);
      \draw[<-] (A1) to node[midway, above] {\scriptsize reward $-1$} (A2);
      \draw[<-] (A2) to node[midway, above] {\scriptsize reward $-1$} (A3);
      \draw[<-] (A3) to node[midway, above] {\scriptsize reward $-1$} (A4);
      \draw[<-] (A4) to node[midway, above] {\scriptsize reward $-1$} (A5);
    \end{tikzpicture}
    \caption{
      Markov chain of a job with \emph{known} service time.
      The job's state is its \emph{remaining service}, namely how many more time units of service until it completes---meaning, transitions to the finished state~$\checkmark$.
      Every transition yields reward $-1$, representing one unit of time passing.
      The transitions are all deterministic, and one can confirm that $\checkmark$ is reached.
    }
    \label{fig:job_markov_chain_known}
  \end{figure}%
}

% Hack around ezlib's `mathic` "feature", which changes $'s catcode...
{\catcode`\$=3%
\gdef\drawTransitionsUnknown#1#2#3#4{%
  % For some reason, the edges have a giant bounding box
  % Using a scope with `overlay` essentially ignores this bounding box
  \begin{scope}[overlay]
    \draw[->] (#10) to node[midway, above] {\scriptsize reward $-1$} (#11);
    \draw[->] (#11) to node[midway, above] {\scriptsize reward $-1$} (#12);
    \draw[->] (#12) to node[midway, above] {\scriptsize reward $-1$} (#13);
    \draw[->] (#13) to node[midway, above] {\scriptsize reward $-1$} (#14);
    \draw[->] (#14) to node[midway, above] {\scriptsize reward $-1$} (#15);
    \draw[->] ($(#10)!0.3!(#11)$) to[out=0, in=70] +(0.2 * \xscale, -0.7) +(0.15 * \xscale, -0.7) node[left] {\scriptsize w.p. $#4{#3_1}{#3_{\geq 1}}$} +(0.2 * \xscale, -0.7) to[out=250, in=5*15+15] (#20);
    \draw[->] ($(#11)!0.3!(#12)$) to[out=0, in=70] +(0.2 * \xscale, -0.7) +(0.15 * \xscale, -0.7) node[left] {\scriptsize w.p. $#4{#3_2}{#3_{\geq 2}}$} +(0.2 * \xscale, -0.7) to[out=250, in=4*15+15] (#20);
    \draw[->] ($(#12)!0.3!(#13)$) to[out=0, in=70] +(0.2 * \xscale, -0.7) +(0.15 * \xscale, -0.7) node[left] {\scriptsize w.p. $#4{#3_3}{#3_{\geq 3}}$} +(0.2 * \xscale, -0.7) to[out=250, in=3*15+15] (#20);
    \draw[->] ($(#13)!0.3!(#14)$) to[out=0, in=70] +(0.2 * \xscale, -0.7) +(0.15 * \xscale, -0.7) node[left] {\scriptsize w.p. $#4{#3_4}{#3_{\geq 4}}$} +(0.2 * \xscale, -0.7) to[out=250, in=2*15+15] (#20);
    \draw[->] ($(#14)!0.3!(#15)$) to[out=0, in=70] +(0.2 * \xscale, -0.7) +(0.15 * \xscale, -0.7) node[left] {\scriptsize w.p. $#4{#3_5}{#3_{\geq 5}}$} +(0.2 * \xscale, -0.7) to[out=250, in=1*15+15] (#20);
  \end{scope}
}

\gdef\figJobMarkovChainUnknown{%
  \begin{figure}[t]
    \begin{tikzpicture}
      \newcommand{\xscale}{2.5}
      \node[state] (A0) at (0 * \xscale, 0) {$0$};
      \node[state] (A1) at (1 * \xscale, 0) {$1$};
      \node[state] (A2) at (2 * \xscale, 0) {$2$};
      \node[state] (A3) at (3 * \xscale, 0) {$3$};
      \node[state] (A4) at (4 * \xscale, 0) {$4$};
      \node[state, draw=none, fill=none] (A5) at (5 * \xscale, 0) {$\dots$};
      \node[state selected] (B0) at (0 * \xscale, -3) {$\checkmark$};
      \drawTransitionsUnknown{A}{B}{m}{\dfrac}
    \end{tikzpicture}
    \caption{
      Markov chain of a job with \emph{unknown} service time sampled from distribution~$m$---that is, the job's service time is $t$ with probability~$m_t$, and it is at least~$t$ with probability $m_{\geq t} = \sum_{u = t}^\infty m_u$.
      The job's state is its \emph{attained service}, namely how many time units of service it has already received.
      Every transition yields reward $-1$, representing one unit of time passing.
      The transition probabilities come from the fact that if the job is in state~$s$, then because the job has not yet completed, its service time must be at least~$s + 1$.
      Given this, a job in state~$s$ completes within its next unit of service with probability $m_{s + 1} / m_{\geq s + 1}$.
    }
    \label{fig:job_markov_chain_unknown}
  \end{figure}
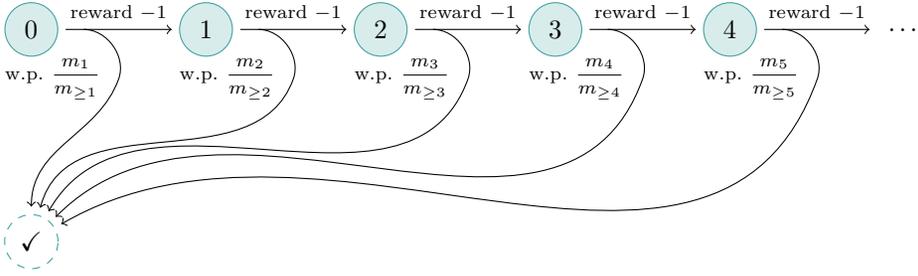%
}%

\gdef\figJobMarkovChainStages{%
  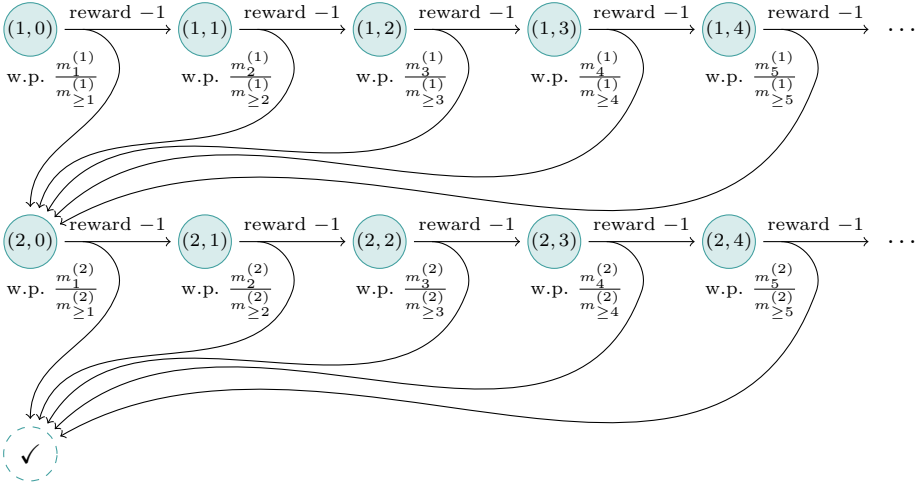
\begin{figure}[t]
    \begin{tikzpicture}
      \newcommand{\xscale}{2.5}
      \node[state, inner sep=0pt] (A0) at (0 * \xscale, 0) {\scriptsize $(1, 0)$};
      \node[state, inner sep=0pt] (A1) at (1 * \xscale, 0) {\scriptsize $(1, 1)$};
      \node[state, inner sep=0pt] (A2) at (2 * \xscale, 0) {\scriptsize $(1, 2)$};
      \node[state, inner sep=0pt] (A3) at (3 * \xscale, 0) {\scriptsize $(1, 3)$};
      \node[state, inner sep=0pt] (A4) at (4 * \xscale, 0) {\scriptsize $(1, 4)$};
      \node[state, draw=none, fill=none] (A5) at (5 * \xscale, 0) {$\dots$};
      \node[state, inner sep=0pt] (B0) at (0 * \xscale, -3) {\scriptsize $(2, 0)$};
      \node[state, inner sep=0pt] (B1) at (1 * \xscale, -3) {\scriptsize $(2, 1)$};
      \node[state, inner sep=0pt] (B2) at (2 * \xscale, -3) {\scriptsize $(2, 2)$};
      \node[state, inner sep=0pt] (B3) at (3 * \xscale, -3) {\scriptsize $(2, 3)$};
      \node[state, inner sep=0pt] (B4) at (4 * \xscale, -3) {\scriptsize $(2, 4)$};
      \node[state, draw=none, fill=none] (B5) at (5 * \xscale, -3) {$\dots$};
      \node[state selected] (C0) at (0 * \xscale, -6) {$\checkmark$};
      \drawTransitionsUnknown{A}{B}{m^{(1)}}{\tfrac}
      \drawTransitionsUnknown{B}{C}{m^{(2)}}{\tfrac}
    \end{tikzpicture}
    \caption{
      Markov chain of a job with \emph{two stages} of service.
      Each stage~$i$ requires an unknown amount of service sampled from distribution~$m^{(i)}$, so the service time is a priori unknown, but the agent learns when the transition from the first to the second stage occurs.
      A job's state is thus a pair consisting of the current stage and the attained service of the current stage.
      The result is, roughly speaking, two stacked instances of the Markov chain from \cref{fig:job_markov_chain_unknown}.
    }
    \label{fig:job_markov_chain_stages}
  \end{figure}%
}%
}

Another famous application of the Gittins policy is scheduling in single-server queues, particularly the M/G/1 queue \citep{sevcik_use_1971, sevcik_scheduling_1974, olivier_kostenminimale_1972, klimov_time-sharing_1974, klimov_time-sharing_1978, aalto_gittins_2009, aalto_minimizing_2023, scully_gittins_2021, bertsimas_optimization_1999, bertsimas_achievable_1995, whittle_tax_2005}.
Here the problem is to schedule jobs that arrive over time on a single server in order to minimize their \emph{mean latency}---also called their response time or sojourn time---which is the mean amount of time between a job arrives and when it completes.

In this section, we give a brief overview of how to formulate job scheduling in terms of MCS.
We first explain the arrival-free \emph{batch setting}, then explain how the theory extends to handle (time-homogeneous) \emph{Poisson arrivals}.
For simplicity, we focus throughout on the unweighted case.
See \citet{scully_gittins_2021} for an account of most of the features the Gittins policy is known to handle in the M/G/1, including unknown and state-varying job weights, and \citet{glazebrook_analysis_2003} for an additional discussion of priority constraints.

\subsubsection{Representing batch scheduling as MCS}

\figJobMarkovChainKnown

Broadly, (single-server, discrete-time) \emph{batch scheduling} is a family of problems where an agent has $n$ \emph{jobs} to complete using a single server.
Every time step, the agent chooses one job to receive one time unit of service.
Each job requires a (strictly positive) number of time units at the server to complete, which we call the job's \emph{service time}.
As we discuss below, these service times might be known, unknown, or partially known to the agent.
The agent's goal is to optimize some metric related to the jobs' \emph{completion times}, where a job's completion time is the time step when it completes.
\footnote{%
  To be precise: letting the first time step have index~$1$, a job's completion time is the index of the time step during which it receives its last unit of service.
  For example, if a job with service time~$t$ is served starting at time~$1$ until it completes, then it receives service at times $1, \dots, t$ and has completion time~$t$.
}
We focus on minimizing the expected values of the following three metrics:
\*[1.] The \emph{earliest completion time} is the minimum of all jobs' completion times.
\* The \emph{$k$th completion time} is the $k$th least among all jobs' completion times.
First completion time is then the special case where $k = 1$.
\* The \emph{total completion time} is the sum of all jobs' completion times.
An equivalent metric is \emph{mean completion time}, which is total completion time times~$1/n$.
These are the metrics most closely related to minimizing mean latency in queues with arrivals.
\*/

In the above, we stated that service times might be known, unknown, or partially known.
Specifically, we assume that \emph{each job's service can be modeled as a Markov chain}, with different jobs' Markov chains evolving independently.
A job's service time is then the number of transitions it takes to reach a \emph{completed} state, which we denote by~$\checkmark$.
The appropriate Markov chain for each job depends on what the agent knows about that job's service time.
For example:
\*[a.] For jobs with known service time, we use the Markov chain in \cref{fig:job_markov_chain_known}.
A job's state here is its \emph{remaining service time}.
With each unit of service, the remaining service time decrements by~$1$, with completed state~$\checkmark$ taking the place of~$0$ remaining service time.
In this case, the Gittins index of a job is simply its (negative) remaining service time
\[
  G(s) = -s
\]
and the Gittins policy reduces to (discretized) shortest remaining processing time \citep{schrage_proof_1968}.
\* For jobs with unknown service time sampled from some known distribution~$m$, we use the Markov chain in \cref{fig:job_markov_chain_unknown}.
A job's state here is its \emph{attained service}.
With each unit of service, the attained service typically increments by~$1$, but there is a chance to transition to the completed state~$\checkmark$.
In this case, the Gittins index can be written~\citep{aalto_gittins_2009, aalto_properties_2011}
\[
  G(s) = -{\smashoperator[l]{\inf_{s' > s}} \frac{\E_{t \sim m}{\min(t, s') - s \given t > s}}{\P_{t \sim m}{t \leq s' \given t > s}}}
\]
where $t \sim m$ is service time of the job, namely the random number of steps it takes to reach~$\checkmark$ starting from~$0$.
The intuition is that $s'$ represents a deadline by which we might hope the job will finish, the fraction is the \emph{mean time per completion} ratio for deadline~$s'$, and $G(s)$ is the (negated) best---namely, least---such ratio achievable.
\* An intermediate case where the agent receives partial information about a job's service time is illustrated in \cref{fig:job_markov_chain_stages}.
Here a job's service has two stages, and while the agent does not know a priori how long either stage will take, the agent is notified when the job advances from the first stage to the second.
\*/
Throughout, we give all states reward~$-1$ to represent the fact that one unit of time passes per transition---though, as we discuss below, this is not essential.
The three examples presented here generalize easily to a large number of possible Markov chain variants.

\figJobMarkovChainUnknown

Having specified the job model we are working with, we can observe the following about the different metrics we hope to optimize.
The main takeaway is that \emph{all the problems are variants of undiscounted MCS}---and, thus, they can be solved with the Gittins policy.
\*[i.] Minimizing expected earliest completion time is MCS (\cref{def:mcs}), because we receive reward~$-1$ per time step until a job completes.
Therefore, by \cref{thm:gittins_optimal}, the Gittins policy minimizes expected earliest completion time.
\* Minimizing expected $k$th completion time is MCS-$k$ (\cref{def:mcs-k}), because we receive reward~$-1$ per time step until $k$ jobs complete.
Therefore, by \cref{thm:gittins_optimal_mcs-k}, the Gittins policy minimizes expected $k$th completion time.
\* Minimizing expected total completion time does not directly fit into MCS or MCS-$k$.
However, we can express total completion time $C$ as the sum of $k$th completion times $C_{(k)}$, namely
\[
  C = \sum_{k = 1}^n C_{(k)}
  .
\]
Because the Gittins policy minimizes $\E{C_{(k)}}$ for all~$k$, it also minimizes~$\E{C}$.
\*/

Before discussing adding arrivals, let us briefly revisit one of our assumptions, namely that every transition incurs cost~$1$ (and thus yields reward~$-1$).
The MCS framework allows different costs in different states, so the entire discussion above generalizes to jobs where different transitions incur different costs.
One can interpret this as scheduling jobs where \emph{different transitions take different amounts of time}.
In particular, transitions that take a long time can be viewed as \emph{uninterruptible segments}.
That is, one can represent an interruptible part of a job using many low-cost transitions, and one can represent a uninterruptible part using a single high-cost transition.

\figJobMarkovChainStages

\subsubsection{From batch scheduling to queue scheduling}

Broadly, a (single-server, discrete-time) \emph{queue scheduling} problem is a batch scheduling problem with the following changes:
\*[1.] Instead of $n$ jobs being present at time~$0$, jobs arrive over time.
We usually consider an infinite sequence jobs arriving according to some stochastic process.
\* Instead of optimizing metrics related to completion time, we optimize metrics related to \emph{latency} (also called response time or sojourn time), where a job's latency is its completion time minus its arrival time.
We usually consider metrics that are (limiting) averages over the sequence of arriving jobs, such as mean latency.
\*/

If new Markov chains can appear over time in an arbitrary fashion, then the Gittins policy is no longer optimal for MCS, let alone MCS-$k$.
However, there is a more general case where the Gittins policy remains optimal: when arrivals are generated by a \emph{memoryless and time-homogeneous} stochastic process \citep{whittle_arm-acquiring_1981}.
This means \emph{Poisson arrival times}, where each arriving job's initial state is drawn i.i.d. from some initial state distribution (or a similar arrival process---see \cref{rmk:mg1_extensions}).
In the language of queueing theory, this is the \emph{M/G/1} queueing model, which allows Poisson arrival times and general service time distributions.
We now illustrate one type of result that can be shown, though we emphasize that it is not the most general result possible \citep{scully_gittins_2021, glazebrook_analysis_2003}.

\begin{definition}
  \label{def:mg1}
  A \emph{Markov chain M/G/1} model consists of:
  \*[1.] A \emph{job Markov chain} $(S, \partial S, p, r)$ such that $r(s) < 0$ for all non-terminal states~$s \in S \setminus \partial S$.
  \* An \emph{initial state distribution} $p_0 \in \calP(S \setminus \partial S)$.
  \* An \emph{arrival rate} $\lambda > 0$.
\*/
\end{definition}

Here, jobs arrive at the increments of a Poisson process of rate~$\lambda$, and each job's initial state is sampled i.i.d. from~$p_0$.
Each job's Markov chain advances while it is in service.
It takes $-r(s)$ time for a job to transition from state~$s$ to its next state (possibly with some randomness---see \cref{rmk:mg1_extensions}), and a job cannot be interrupted during a state transition.
A \emph{scheduling policy} for the Markov chain M/G/1 repeatedly chooses which job to serve next, at which point the job remains in service until it has completed a state transition.
We also allow the scheduling policy to leave the server idle for a time.

We say the M/G/1 is \emph{stable} if the expected interarrival time---that is, $1/\lambda$---is greater than the expected service time---that is, the amount of time it takes a job to transition from an initial state drawn from $p_0$ to a terminal state in~$\partial S$.
This ensures ergodicity of the system under any (non-idling) scheduling policy, in which case the following optimality statement holds for the Gittins policy.

\begin{theorem}
  \label{thm:gittins_optimal_mg1}
  In any stable Markov chain M/G/1, among all scheduling policies, the Gittins policy minimizes mean latency.
\end{theorem}

Unfortunately, to the best of our knowledge, there are no proofs of \cref{thm:gittins_optimal_mg1} that build directly on the techniques used to prove \cref{thm:gittins_optimal, thm:gittins_optimal_mcs-k}.
Proofs using a \emph{vanishing discount} approach---see \citet[Chapter~3]{gittins_multi-armed_2011}---come the closest, but they require restrictive assumptions, such as for instance the job Markov chain having a finite state space.
Thus far, the approaches that yield the most general results use \emph{work conservation laws} or similar M/G/1-specific reasoning \citep{bertsimas_optimization_1999, scully_gittins_2021, glazebrook_analysis_2003}.

\begin{remark}
\label{rmk:mg1_extensions}
The optimality of the Gittins policy continues to hold under the following two types of extensions:
\*[subenv] One can generalize slightly beyond Poisson arrivals to other time-homogeneous arrival processes.
The simplest case of this is \emph{batch Poisson} arrivals, where at each arrival time, multiple jobs can arrive at once \citep{scully_gittins_2021}.
In this batch setting, the lists of initial states of jobs in a batch must be i.i.d. across batches.
Another time-homogeneous case occurs if all jobs are time-slotted---meaning, if all jobs' transitions take an integer number of time unit---then one can work with arrival processes that are similarly time-slotted, even if they are not Poisson.
\* One can generalize from deterministic to randomized transition times in job Markov chains, in which case $-r(s)$ should be the \emph{expected} amount of time it takes to transition from $s$ to the next state.
In this more general model, instead of specifying a job via $p$ and $r$, one specifies a kernel $q : S \to \calP(S \times [0, \infty))$, where $q(s)$ is the joint distribution over $(\text{next state}, \text{transition time})$ pairs.
Even though one $(p, r)$ pair can arise from many possible kernels~$q$, the Gittins index depends only on $p$ and~$r$, though the mean latency achieved is affected by~$q$.
  This generalization is possible because it embeds into the continuous-time framework of \citep{scully_gittins_2021}.
\*/
\end{remark}

Of course, the M/G/1 is a rather specialized queueing model, so it is natural to ask whether the Gittins policy performs strongly in more general settings---say, with more than one server or with non-Poisson arrivals.
Thus far, it appears the answer is often \emph{yes}: there are now several results showing the Gittins policy is in some sense approximately optimal, including the multiserver M/G/$k$ \citep{glazebrook_parallel_2001, glazebrook_analysis_2003, scully_gittins_2020, scully_new_2022, grosof_optimal_2022} and, most recently, the non-Poisson G/G/1 and G/G/$k$ \citep{hong_performance_2024}.
Even under adversarial arrivals---in the sense of adversarially chosen arrival times and initial states, but where the job still evolves stochastically according to a known transition kernel during service---it is known that a variation of the the Gittins policy is a $2$-approximation for mean \emph{completion time} \citep{megow_tight_2014}, meaning each job's clock starts at time~$0$.
It is an open question whether a similar guarantee holds for mean \emph{latency}, where each job's clock starts when it arrives.

\subsection{Branching bandits}
\label{sec:examples:branching}

To arrive at a general formulation of the Gittins index, following our exposition of Pandora's box, we asked: \emph{is there a general theory this solution is an example of?}
In the same spirit, let us note that the Gittins index for the M/G/1 queue does not directly arise as an instance of MCS, and again ask the same question.

Compared to MCS, what is new about the M/G/1 setting is that every time an action is chosen, a new set of Markov chains arrives, according to the Poisson process and initial state distribution, and becomes part of the total action space.
More precisely, when a Markov chain in state~$s$ transitions to a new state $s' \sim p(s)$, it also gives rise to a random set of new Markov chains in some initial states drawn from a known distribution.

The appropriate abstract generalization of the M/G/1 setting is the so-called \emph{branching bandit} setting of \citet{weiss_branching_1988}.
The branching bandit problem is like MCS, but it replaces each Markov chain with a particular kind of \emph{branching process} (namely, a multi-type Galton--Watson process), defined in the following sense.
When advancing a branching process, instead of its state~$s$ transitioning to exactly one next state, it transitions by replacing the current state with multiple new states according to a suitable probability kernel, which describes the joint distribution over the number of new states and their values.
This means that the overall problem's state space is now described by tuples $(s_1,..,s_n)$ of variable length, where $n$ varies according to the actions selected by the policy and the random transition outcomes.

Variants of the branching bandit problem have been studied in the discounted setting \cite{gittins_multi-armed_2011, weiss_branching_1988, meister_optimizing_2023}, as well as the undiscounted Pandora's box setting \cite{boodaghians_pandoras_2020}: in these cases, one can define a Gittins index in an appropriate sense, and prove that the resulting Gittins policy is optimal.
However, to the best of our knowledge, there is not yet a proof that holds at the level of generality of \cref{thm:gittins_optimal}.
For instance, \citet{weiss_branching_1988} requires each Markov chain to have finitely many states and satisfy \cref{asm:discount:discounted}.
In principle, a more general dynamic programming proof similar to the one of \cref{sec:proof} should be feasible, but we are not aware of one.\footnote{
    In the language of \cref{sec:beyond:optional_inspection:progress}, we believe that one should be able to show that all branching processes satisfy the Whittle condition in an appropriate sense, and use this to construct an optimality argument.
}

\section{Examples: beyond optimality}
\label{sec:beyond}

We now discuss three examples where Gittins indices can be defined and applied, but do not result in an optimality proof.
These are (a)~certain forms of Bayesian optimization, where optimality fails due to the presence of correlations between different Markov chains (\cref{sec:beyond:bayesopt}); (b)~Pandora's box with optional inspection, where additional decisions that can be made render the problem more complicated (\cref{sec:beyond:optional_inspection}); and (c)~minimizing tail latency in queues, where one seeks to perform well in terms of objectives beyond average rewards (\cref{sec:beyond:tail}).

\subsection{Bayesian optimization}
\label{sec:beyond:bayesopt}

\newcommand{\figBayesoptResults}{%
  \begin{figure}[t]
    \includegraphics{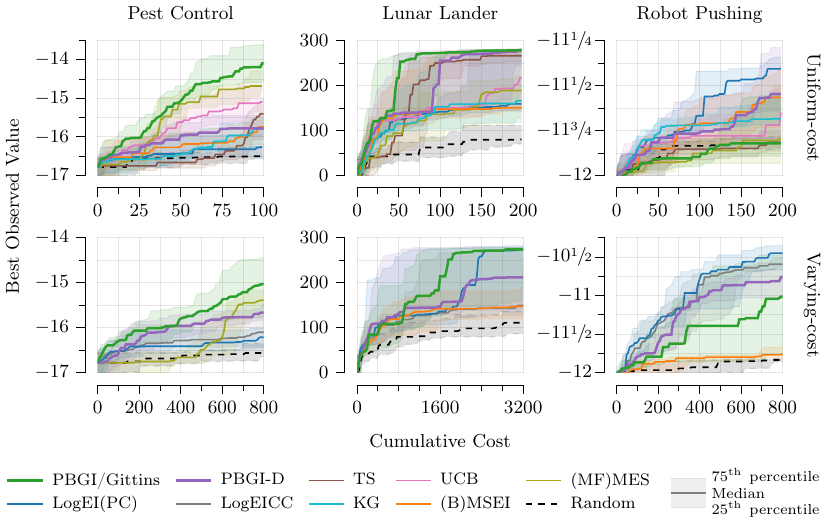}
    \caption{
      Results reproduced from \citet{xie_cost-aware_2024}: empirical performance (higher is better) of the Gittins policy for Bayesian optimization, also called \emph{PBGI} (green) in the legend, against other baseline policies, shown in terms of medians and quartiles over 16 seeds.
      The task here is to optimize benchmark objective functions, constructed to resemble real-world black-box optimization settings.
      We plot the best observed function value, in terms of medians and quartiles from 16 trials with different random initializations.
      We see that PBGI, along with a minor variant called PBGI-D (purple), claim the top-performing spot in the first two problems, and are reasonably competitive in the third.
      This holds for both $c(x)$ chosen to be a constant function, termed the \emph{uniform-cost} setting, and $c(x)$ non-constant, termed the \emph{varying-cost} setting.
    }
    \label{fig:bayesopt_results}
  \end{figure}%
}

\emph{Bayesian optimization} \citep{frazier_bayesian_2018, garnett_bayesian_2023} is a broad class of algorithms for global optimization of unknown functions which are expensive to evaluate.
In most instances, such algorithms require only \emph{black-box access} to the unknown function, meaning the only way to learn about the function is by evaluating it.
Bayesian optimization is a workhorse tool in areas like machine learning hyperparameter tuning \cite{snoek_practical_2012}, where it is deployed in production at most major technology companies, and is available as standard functionality in popular artificial intelligence operations platforms.

Let $f : X \to \bbR$ be the unknown function, where $X$ is allowed to be a general set, with $X = [0,1]^d$ a typical choice.
Bayesian optimization works by building a \emph{probabilistic model} for the unknown function $f$, typically in a Bayesian manner by placing a Gaussian process prior over it.
Function evaluations are incorporated into the probabilistic model using Bayes' Rule, by conditioning the prior on the location and value of previous function evaluations.
With this setup, one aims to design a policy that adaptively chooses inputs $x_1, \dots, x_T$ in order to find the global optimum in as few function evaluations as possible.

There are multiple distinct criteria one use to can study performance of Bayesian optimization algorithms.
The simplest, arguably, is expected \emph{simple regret} with respect to the prior, namely
\[
\E[\bigg]{\sup_{x\in X} f(x) - f^*_T}
\]
where $f^*_t = \max_{u \in \{1, \dots, t\}} f(x_u)$ and the unknown function $f$ is sampled from the prior.
In this setting, the question of how to adaptively choose the next data point $x_{t+1}$, given a set of previous function evaluations $(x_1,f(x_1)),\dots,(x_t,f(x_t))$ can be formalized to define an MDP.
Solving this MDP---and, indeed, almost all other MDPs which occur in Bayesian optimization contexts---is known to be intractable.
However, by applying a one-step greedy approximation to this MDP's dynamic programming equations, one arrives at the \emph{expected improvement acquisition function}
\[
\label{eq:bayesopt_ei_acquisition}
\EI_t(x) = \E{\max(f(x) - f^*_t, 0)}
.
\]
A very similar expected improvement formula above previously came up in \cref{sec:pandora}.
This is not a coincidence: consider a mild generalization known as \emph{cost-aware Bayesian optimization}, specifically the \emph{cost-per-sample} formulation, where one adds a sequence of costs $c(x_t) > 0$ to the above objective, and allows the algorithm to decide when to stop, as opposed to having $T$ be a fixed hyperparameter.
Note that that $\E{\sup_{x\in X} f(x)}$ is constant with respect to the policy.
Using this, following \citet{xie_cost-aware_2024}, if we switch from minimizing regret to maximizing the negation of all terms, and drop the aforementioned constant, we obtain
\[
\E[\Bigg]{f^*_T - \sum_{t=1}^T c(x_t)}
\]
which is the same as the objective of Pandora's box in \cref{eq:pandora_objective}.
We therefore conclude:
\*[{(\ensuremath{\star})}] \emph{Cost-per-sample Bayesian optimization, under the expected simple regret performance criterion, is a Pandora's box problem with correlations between boxes.}
\*/
Compared to the Pandora's box of \cref{sec:pandora}, the crucial difference here is that there is now a potentially uncountable number of boxes, indexed by $X$, and rewards in different boxes are correlated.
These correlations completely break the argument of \cref{thm:gittins_optimal}, which no longer applies.
On the other hand, the same correlations also make the optimal policy intractable: one can therefore ask whether the Gittins policy at least makes sense as a candidate to consider implementing in practice.

\subsubsection{Defining a Gittins index for Bayesian optimization}
\label{sec:beyond:bayesopt:definition}

To proceed, we view cost-per-sample Bayesian optimization as a variant of MCS where whenever one Markov chain advances, \emph{the transition kernels of all other Markov chains also update}.
In Bayesian optimization, this transition kernel update occurs due to updating the posterior distribution of the unknown function~$f$.
The appropriate definition of a Gittins policy for this variant of MCS---and thereby for Bayesian optimization---is just like the Gittins policy for ordinary MCS, but where we make sure to \emph{always use the updated transition kernel when defining Gittins indices}.
We now make this precise.

Suppose we have already observed the values $f(x_1), \dots, f(x_t)$, resulting in a posterior distribution $p_t$ for~$f \mid f(x_1), \dots, f(x_t)$.
Based on this posterior distribution, we define the \emph{Gittins index of an input point~$x$} in much the same way as the Gittins index of a Pandora's box in \cref{sec:pandora:gittins}.
Specifically, we think of input point~$x$ as corresponding to a box whose opening cost is~$c(x)$ and whose value distribution is that of $f(x)$ for $f \sim p_t$:
That is, we define the expected improvement function at time~$t$ to be
\[
  \label{eq:bayesopt_ei}
  \EI_t(x, \alpha) &= \E_{f \sim p_t}{\max(f(x) - \alpha, 0)} - c(x)
  \\
  &= (\mu_t(x) - \alpha) \Phi\left(\frac{\mu_t(x) - \alpha}{\sigma_t(x)}\right) + \sigma_t(x)\phi\left(\frac{\mu_t(x) - \alpha}{\sigma_t(x)}\right) - c(x)
\]
where $\mu_t$ and $\sigma_t$ are the mean and standard deviation of the posterior Gaussian process, and $\phi$ and $\Phi$ are the standard Gaussian PDF and CDF, respectively.
From this, we define the Gittins index function at time~$t$ to be $G_t$, where $G_t(x)$ solves the root-finding problem
\[
  \label{eq:bayesopt_gittins}
  \EI_t(x, G_t(x)) = 0
  .
\]
Here, the expected improvement function in \cref{eq:bayesopt_ei_acquisition} is the special case of \cref{eq:bayesopt_ei} where we always plug in $\alpha = f^*_t$.
At time~$t$, meaning after evaluating $f(x_1), \dots, f(x_t)$, we need to choose between one of the following actions.
\*[a.] Choose a next point $x_{t + 1}$ to evaluate.
As described above, choosing to evaluate $x_{t + 1} = x$ has Gittins index $G_t(x)$.
\* Choose to stop by setting $T = t$ and taking an observed value, namely $f^*_T = \max(f(x_1), \dots, f(x_T))$.
This action yields reward $f^*_T$ and ends the process---so, just as in ordinary Pandora's box, it has Gittins index~$f^*_T$.
\*/
The \emph{Gittins policy for Bayesian optimization}, also called the \emph{Pandora's box Gittins index (PBGI)} by \citet{xie_cost-aware_2024} to emphasize its connection with Pandora's box, is then the policy that always takes the action of maximum Gittins index.

One point of subtlety is that, even though \cref{eq:bayesopt_ei, eq:bayesopt_gittins} are essentially the same as their Pandora's box counterparts in \cref{eq:pandora_ei, eq:pandora_gittins}, which involve no correlations, it is \emph{not correct} to say that \cref{eq:bayesopt_ei, eq:bayesopt_gittins} ignore correlations.
This is because they use the posterior distribution~$p_t$, which accounts for correlations in its definition.
Thus, a more accurate description would be to say that \cref{eq:bayesopt_ei, eq:bayesopt_gittins}, in some sense, account for \emph{correlations from the past}, but disregard \emph{correlations in the future}.

\subsubsection{Maximizing the Gittins index numerically}
\label{sec:beyond:bayesopt:computation}

Since the state space is infinite, maximizing $G_t(x)$ in order to select the next data point requires gradient-based optimization.
In Bayesian optimization, this step is called \emph{acquisition function optimization}, and is generally performed numerically using multi-start variants of either LBFGS or ADAM.
The challenge now is that one needs to compute the gradient $\nabla G_t(x)$, which involves automatically differentiating through the root-finding problem.
To avoid differentiating through the individual steps of bisection search or other root-finding algorithm, \citet{xie_cost-aware_2024} show that $\nabla G_t(x)$ admits a particularly simple form, namely
\[
\label{eq:bayesopt_gittins_gradient}
\nabla G_t(x) = \nabla \mu(x) + \frac{\phi\left(\frac{\mu_t(x) - G_t(x)}{\sigma_t(x)}\right) \nabla \sigma_t(x) - \nabla c(x)}{\Phi\left(\frac{\mu_t(x) - G_t(x)}{\sigma_t(x)}\right)}
.
\]
Using this, one can compute $G_t(x)$ numerically using bisection search, then plug the result in to \cref{eq:bayesopt_gittins_gradient} to obtain the gradient.
This approach is an instance of a general principle used throughout the automatic differentiation literature \cite{agrawal_differentiable_2019, blondel_efficient_2022}: one can differentiate through the solution of a root-finding problem numerically by expressing the respective derivative in terms of the function defining the root-finding problem, together with the root.

\figBayesoptResults

\subsubsection{Performance of the Gittins index in Bayesian optimization}

With an appropriate Gittins policy---which is not optimal---defined, the question becomes: is it strong?
Empirically, at least for the kind of Gaussian processes which are used in Bayesian optimization benchmarking, the answer appears to be \emph{yes}: \citet{xie_cost-aware_2024} show that the Gittins policy either matches or outperforms most baselines---this is shown in \cref{fig:bayesopt_results}.

This connection appears to be new: to the best of our knowledge, it remained unnoticed until the recent work of \citet{xie_cost-aware_2024}, and prior to that, only \citet{persky_exploration_2021} had approached Bayesian optimization using a discounted-bandit variant of Gittins indices.
One benefit the Gittins index perspective brings to Bayesian optimization is that Gittins indices naturally handle different input points having different function evaluation costs---an ongoing challenge in Bayesian optimization \citep{xie_cost-aware_2024, lee_efficient_2020, lee_nonmyopic_2021}---because the Pandora's box problem naturally allows different boxes to have different opening costs.

Developing a theory that characterizes the strengths and limitations of the Gittins policy's performance in Bayesian optimization---for instance, in the language of regret or approximation ratio bounds---is the subject of ongoing research.
Recent results on Pandora's box with general joint value distributions \citep{gergatsouli_weitzmans_2023, chawla_approximating_2022} may provide a good starting point, though it is likely that stronger guarantees might be possible when focusing on the multivariate Gaussian distributions that arise in typical Bayesian optimization priors.

Another rich future direction is developing versions of the Gittins policy for more advanced Bayesian optimization settings, such as \emph{multi-fidelity optimization} for applications like hyperparameter tuning \citep{zimmer_auto-pytorch_2021, eggensperger_hpobench_2021}.
These are problems where there are multiple actions one can take at any given input point---for instance, one can either fully evaluate the function, or obtain a cheap but noisy value estimate.
Some such problems may share features of our next example: a Pandora's box variant with two actions available for closed boxes.

\subsection{Pandora's box with optional inspection}
\label{sec:beyond:optional_inspection}

A natural question about the Pandora's box problem is: what changes if one is allowed to select a closed box---without opening it first?
This variant of the problem is called Pandora's box with \emph{optional inspection} (also known as nonobligatory inspection) \citep{doval_whether_2018}, in contrast with the original problem's \emph{required inspection} (also known as obligatory inspection).
Optional inspection results in a much harder problem than required inspection, with \citet{fu_pandora_2023} showing it is NP-hard in an appropriate computational sense.
In particular, while the Gittins index can still be defined, the Gittins policy is no longer optimal for reasons we explain below.
With this said, the Gittins index is still a critical tool for the optional inspection setting: a number of approximation algorithms for Pandora's box with optional inspection are known \citep{scully_local_2024, chawla_combinatorial_2024, beyhaghi_pandoras_2023, fu_pandora_2023, guha_information_2008, beyhaghi_pandoras_2019}, and all of them use the Gittins policy as a subroutine.

\newcommand{\figPandoraOptionalMdp}{%
  \begin{figure}[t]
    \begin{tikzpicture}
      \drawBoxClosed{-6.5, 0.5}{$\boxclosed$}

      \coordinate (A) at (-2.5, 0.5);
      \draw (-5.5, 0.5) -- (A)
        node[midway, above] {$\op$}
        node[midway, below] {\scriptsize reward $-c$, next state $\sim p$}
        -- +(0.01, 0);

      \node at (-1.625, +3) {$\vdots$};
      \draw[<-] (-1.25, +2) -- +(-0.25, 0) to[out=180, in=0] (A);
      \draw[<-] (-1.25, +1) -- +(-0.25, 0) to[out=180, in=0] (A);
      \draw[<-] (-1.25, 0) -- +(-0.25, 0) to[out=180, in=0] (A);
      \draw[<-] (-1.25, -1) -- +(-0.25, 0) to[out=180, in=0] (A);
      \node at (-1.625, -1.75) {$\vdots$};

      \node at (0, +3) {$\vdots$};
      \drawBoxOpen{0, +2}{$3.2$}
      \drawBoxOpen{0, +1}{$3.1$}
      \drawBoxOpen{0, 0}{$3.0$}
      \drawBoxOpen{0, -1}{$2.9$}
      \node at (0, -1.75) {$\vdots$};

      \coordinate (C) at (5, 0.5);
      \draw[<-] (5.25, 0.5) -- (C) -- +(-0.01, 0);

      \node at (1.625, +3) {$\vdots$};
      \draw (1.25, +2) -- +(2.25, 0)
        node[midway, above] {$\tk$}
        node[midway, below] {\scriptsize reward $3.2$}
        to[out=0, in=180] (C);
      \draw (1.25, +1) -- +(2.25, 0)
        node[midway, above] {$\tk$}
        node[midway, below] {\scriptsize reward $3.1$}
        to[out=0, in=180] (C);
      \draw (1.25, 0) -- +(2.25, 0)
        node[midway, above] {$\tk$}
        node[midway, below] {\scriptsize reward $3.0$}
        to[out=0, in=180] (C);
      \draw (1.25, -1) -- +(2.25, 0)
        node[midway, above] {$\tk$}
        node[midway, below] {\scriptsize reward $2.9$}
        to[out=0, in=180] (C);
        \node at (1.625, -1.75) {$\vdots$};

      \drawBoxSelected{6.5, 0.5}{$\checkmark$}

      \draw[orange!85!black, ultra thick, ->] (-6.125, 1.25) to[out=45, in=135]
        node[midway, above=-0.1em] {\small\bfseries $\tk$ (new!)}
        node[midway, below=0.1em] {\scriptsize reward $\E_{v \sim p}{v}$}
        (6.125, 1.25);
    \end{tikzpicture}
    \caption{
      A single Pandora's box with optional inspection is not simply a Markov chain, but an MDP: specifically, from the closed state~$\boxclosed$, one can take either the $\op$ action, which incurs cost but reveals the box's reward, or the $\tk$ action, which takes the box without opening it first.
      Once a box is opened, the only available action is $\tk$.
    }
    \label{fig:pandora_optional_mdp}
  \end{figure}
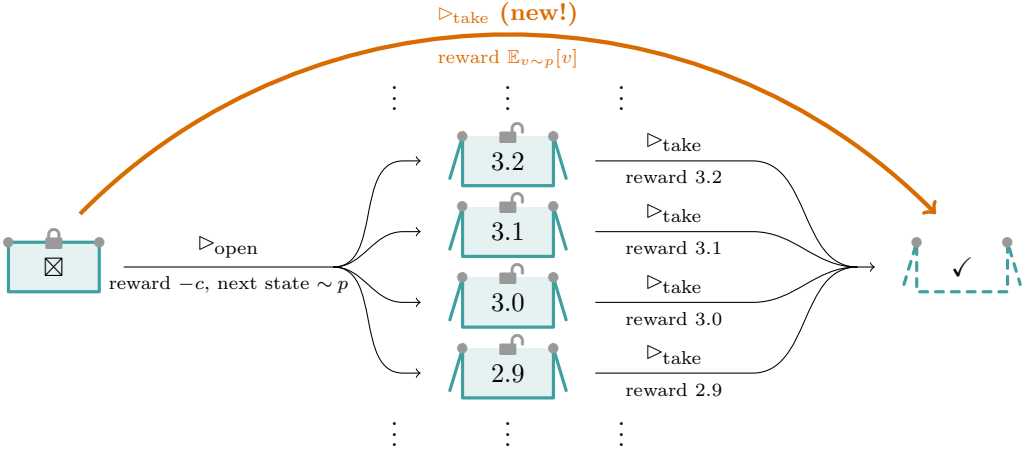%
}
\figPandoraOptionalMdp

The core reason why the Pandora's box problem with optional inspection is difficult is that it is \emph{not} an instance of Markov chain selection (MCS; see \cref{def:mcs}).
Instead, it is an instance of what we call \emph{MDP selection}, where instead of choosing which one of multiple Markov chains to advance at each time step, one chooses which one of multiple MDPs to advance, \emph{along with which action to take in the chosen MDP}.
MDP selection is usually referred to in the Gittins index literature as a the Markovian multi-armed bandit with \emph{bandit superprocesses} \citep{whittle_multi-armed_1980, nash_optimal_1973, brown_optimal_2013, hadfield-menell_multitasking_2015}, but we introduce the MDP selection name for consistency with MCS.
(See also the discussion in \cref{sec:general:mcs:name}.)

The reason Pandora's box with optional inspection is an instance of MDP selection, as opposed to the simpler MCS, is that each box admits two possible actions when it is closed:
\*[a.] \emph{Open}, denoted $\op$ behaves like the Markov chain for required inspection, yielding reward~$-c$ and advances the box to an open state $v \sim p$.
\* \emph{Take}, denoted $\tk$, takes the box without opening it, yielding (expected) reward $\E_{v \sim p}{v}$.
\*/
\cref{fig:pandora_optional_mdp} gives an illustration.

\subsubsection{Why MDP selection is harder than Markov chain selection}

\newcommand{\figureLocalMdpValueOptional}{%
  \begin{figure}[t]
    \hfill%
    \begin{subfigure}{0.475\linewidth}
      \pictureLocalMdpValuePandoraOptional{6}{12}{1}{7}
      \caption{Local MDP value for box~$1$ from \cref{fig:pandora_example}.}
    \end{subfigure}%
    \hfill%
    \begin{subfigure}{0.475\linewidth}
      \pictureLocalMdpValuePandoraOptional{2.6}{13}{2}{3.6}
      \caption{Local MDP value for box~$2$ from \cref{fig:pandora_example}.}
    \end{subfigure}%
    \hfill%
    \caption{
      Analogue of \cref{fig:local_mdp_value} for Pandora's box with optional inspection.
      Value functions of the $(\boxclosed, \alpha)$-local MDP for the two closed boxes in \cref{fig:pandora_example} with three different initial actions: $\op$ (teal), $\tk$ (orange), and $\st$ (yellow-green).
      As in \cref{fig:local_mdp_value}, the optimal action is $\st$ for all values of $\alpha$ above a threshold, and we define the Gittins index $G_i(\boxclosed)$ to be that threshold.
      In these cases, the action that is co-optimal with $\st$ when $\alpha = G_i(\boxclosed)$ is $\op$.
      But when $\alpha$ is lower than another threshold~$H_i$, the optimal first action is $\tk$.
      This means that in the context of a larger Pandora's box with optional inspection problem, or more generally MDP selection, even if we are confident about wanting to play an action on box~$i$, whether we prefer $\op$ or $\tk$ may depend on the states of the other MDPs.
    }
    \label{fig:local_mdp_value_optional}
  \end{figure}%
}

One might hope that the Gittins index approach might extend from MCS to MDP selection.
Indeed, it turns out that one can still define the Gittins index of an MDP in essentially the same way as for a Markov chain, namely using a local MDP (\cref{def:local_mdp}).
Unfortunately, the resulting Gittins policy for MDP selection is generally not optimal.
This is not a surprise: MDP selection is NP-hard, thanks to the aforementioned NP-hardness of Pandora's box with optional inspection \citep{fu_pandora_2023}, so we should not expect the Gittins policy---which can be computed in polynomial time in the finite-state case \citep{balakrishnan_multi-armed_2014, gast_testing_2023, katehakis_multi-armed_1987}---to solve it.
But what specifically prevents the Gittins policy from being optimal?

Let us first clarify what the local MDP and Gittins index look like for an MDP with action space~$A$ instead of a Markov chain.
Just like in \cref{def:local_mdp}, the local MDP is essentially the original MDP with an extra action $\st$ that terminates the process and yields reward~$\alpha$, meaning that we have
\[
A_{\loc} = A \cup \{\st\}
.
\]
But this action space is no longer a two-element set: instead of a single action $\go$ that advances the Markov chain, now each of the MDP's actions $a \in A$ advances it \emph{using action~$a$}, which must be specified.
In spite of this, one can show that \cref{def:gittins_general} continues to make sense, yielding a well-defined Gittins index~$G(s)$ of each state~$s$ among MDPs.
We can then define the Gittins policy as the policy that always plays an action from the MDP of greatest Gittins index $G(s)$, choosing an action, other than $\st$, that is optimal for the $(G(s), s)$-local MDP.
\figureLocalMdpValueOptional

The core issue is that optimality of the Gittins policy relies on the following fact about the local MDP with a Markov chain and any fixed starting state~$s$:
\begin{starquote}
  If the $\go$ action is optimal under some alternative~$\alpha$, then the same $\go$ action is also optimal with any worse alternative $\alpha' < \alpha$.
\end{starquote}
This property fails in general when using an MDP instead of a Markov chain, because the single action $\go$ is replaced by the MDP's action space~$A$.
In particular, the optimal action for the local MDP with alternative~$G(s)$ might be different than the optimal action with lower alternative option.
Intuitively, this is a problem because it means that in full MDP selection, the optimal action to take within one MDP might depend on the states of the other MDPs.
See \cref{rmk:mdps_failure} for details on exactly where the proof of \cref{thm:gittins_optimal} breaks down when generalizing from MCS to MDP selection.

For example, consider the MDP of a box in Pandora's box with optional inspection (\cref{fig:pandora_optional_mdp}) in the closed state~$\boxclosed$.
We show the values in the local MDP for three different initial actions in \cref{fig:local_mdp_value_optional}.
\*[(a)] \label{itm:pandora_optional_stop}
For sufficiently large values of~$\alpha$, as usual, the optimal action is $\st$.
\* \label{itm:pandora_optional_open}
For intermediate values of~$\alpha$, the optimal action is $\op$.
The intuition is that the box has a good chance of being either significantly greater or significantly less than~$\alpha$, so it is worth paying the opening cost to learn the box's value.
\* \label{itm:pandora_optional_take}
For sufficiently small values of~$\alpha$, the optimal action is $\tk$.
The intuition is that the alternative~$\alpha$ is so low that we are very unlikely to prefer it to the box's value, so we are happy taking the box without paying the cost to open it.
\*/
This means that in the context of a broader MDP selection instance with multiple boxes, the optimal $\op$ vs. $\tk$ choice within one box's MDP might depend on the states of the other MDPs.
This can cause prioritizing by Gittins index to be suboptimal: \citet{doval_whether_2018} gives a concrete example.

\subsubsection{Approximate solutions to MDP selection using the Gittins index}
\label{sec:beyond:optional_inspection:progress}

Despite the above challenges, many approximation algorithms have been proposed for Pandora's box with optional inspection \citep{guha_information_2008, beyhaghi_pandoras_2019, beyhaghi_pandoras_2023, fu_pandora_2023, scully_local_2024, chawla_combinatorial_2024}, and similarly for other Pandora's box variants and applications \citep{aouad_pandoras_2020, chawla_combinatorial_2024, beyhaghi_recent_2023, kleinberg_descending_2016, bowers_matching_2024}.
The Gittins index plays a critical role in most of these algorithms.
For example, \citet{fu_pandora_2023} and \citet{beyhaghi_pandoras_2023} show that the optimal policy for the optional inspection setting is a two-phase policy, the second phase of which is to use the Gittins policy; and, while optimally choosing the phase boundary is intractable, they use this insight to construct a polynomial-time approximation scheme for the problem.

There are a few sufficient conditions under which the Gittins policy is known to be \emph{optimal} for MDP selection.
\Citet{doval_whether_2018} identifies one such condition for Pandora's box with optional inspection.
In the general MDP selection setting, \citet{whittle_multi-armed_1980} identifies a condition, now called the \emph{Whittle condition} \citep{glazebrook_sufficient_1982, brown_optimal_2013, hadfield-menell_multitasking_2015} that can be checked separately for each local MDP, with the Gittins policy being optimal if all local MDPs satisfy it.

An MDP satisfies the Whittle condition if, roughly, it can be reduced to a Markov chain with no loss of value in the local MDP.
Specifically, it requires that in every state~$s$ of the MDP, there is a single action~$a$ such that for all alternative values~$\alpha$, either $\st$ or~$a$ is optimal in the $(s, \alpha)$-local MDP.
This precludes the cases shown in \cref{fig:local_mdp_value_optional}, where either $\op$ or $\tk$ can be optimal.
One can show, in Pandora's box with optional inspection, that a box MDP (\cref{fig:pandora_optional_mdp}) satisfies the Whittle condition only in the trivial case where its opening cost is so large that $\op$ is never optimal in the local MDP \citep{doval_whether_2018}.
The Whittle condition is thus relatively limited in scope, though there are some notable classes of MDPs that satisfy it \citep{weiss_branching_1988, glazebrook_stoppable_1979, glazebrook_sufficient_1982}.

However, recent work has revealed fresh promise for the idea of reducing MDPs to Markov chains as a general approach for solving MDP selection: \citet{scully_local_2024} and \citet{chawla_combinatorial_2024} introduce a relaxation of the Whittle condition called \emph{local $\approxRatio$-approximation} and show that many MDPs that fail the Whittle condition admit local approximations.
Roughly speaking, an MDP admits a local $\approxRatio$-approximation if it can be reduced to a Markov chain such that \emph{if the rewards are then scaled by~$\approxRatio$}, there is no loss of value in the local MDP \emph{with any alternative~$\alpha$}, as compared to the original local MDP with the same alternative~$\alpha$ (and without any scaling).
This slightly unusual approximation requirement---which is \emph{not} equivalent to simply achieving a $\approxRatio$-approximation in the local MDP---ensures the following guarantee: in an MDP selection instance where all the MDPs admit (possibly randomized) local $\approxRatio$-approximations, the Gittins policy is a $\approxRatio$-approximation of the optimal policy \citep{scully_local_2024, chawla_combinatorial_2024}. Moreover, this guarantee also holds in the $k$-finish and combinatorial settings described in \cref{sec:examples:finish_multiple} \citep{scully_local_2024, chawla_combinatorial_2024}.
We suspect that local approximation is related to the results of \citet{clarkson_fast_2020}, who prove an approximation guarantee for a special case of MDP selection without explicitly reasoning using the local MDP.

The local approximation approach described above reduces MDPs to Markov chains by attempting to solve the local MDP in a way that is in some sense good for any alternative value~$\alpha$.
A complementary approach to could be to figure out, based on the specific MDP selection instance, what alternative value~$\alpha_i$ is, in some sense, most relevant for each individual MDP~$i$, then take actions within MDP~$i$ that would solve its local MDP with alternative~$\alpha_i$.
\Citet{bowers_prophet_2025} prove a $1/2$-approximation result that is a first step in this direction.
In fact, they obtain their result even under an additional \emph{take it or leave it}    constraint, so it is possible that there are even stronger guarantees waiting to be shown with this approach.

\subsection{Minimizing tail latency in queues}
\label{sec:beyond:tail}

\newcommand{\figTailResults}{%
  \begin{figure}[t]
    \includegraphics{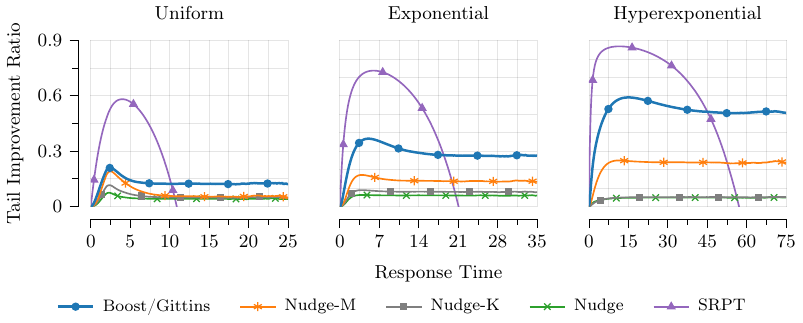}
    \caption{
      Results reproduced from \citet{yu_strongly_2024}: empirical performance (higher is better) of the Gittins policy for minimizing tail probabilities, also called \emph{Boost} (blue) in the legend, against other baseline policies, simulated in three different M/G/1 models with different service time distributions.
      The job model is a continuous-time analogue of the known-size Markov chain from \cref{fig:job_markov_chain_known}.
      The metric shown is \emph{tail improvement ratio} relative to \emph{first-come first-served (FCFS)}, which for policy~$\pi$ and response time threshold~$t$ is $1 - \P{L_\pi > t} / \P{L_{\mathrm{FCFS}} > t}$.
      The probabilities $\P{L_\pi > t}$ are approximated by simulating the policies on a trace of 50 million randomly generated arrivals.
      The Gittins policy is the clear winner over the Nudge family of baselines \citep{grosof_nudge_2021, van_houdt_stochastic_2022, charlet_tail_2024}, with larger improvement when the service time distribution's coefficient of variation is larger.
      \emph{Shortest remaining processing time (SRPT)} (purple) performs better than the Gittins policy for small thresholds~$t$, but SRPT's performance suddenly collapses as $t$ increases.
      This is because although SRPT minimizes mean latency \citep{schrage_proof_1968}, under light-tailed service times, it has the worst possible asymptotic tail probabilities as $t \to \infty$ \citep{nuyens_large-deviations_2006, nuyens_preventing_2008}.
    }
    \label{fig:tail_results}
  \end{figure}
}

As our last example, we revisit the queue scheduling setting from \cref{sec:examples:queue}, where we saw in \cref{def:mg1, thm:gittins_optimal_mg1} that the Gittins policy minimizes \emph{mean latency} in the M/G/1 queue.
However, in many settings, a more relevant objective than minimizing mean latency is minimizing \emph{tail latency}.
Tail latency is a broad term that refers to one of a number of related metrics that capture how likely jobs are to have especially large latency.
Optimizing tail latency is of direct importance to efficiently meeting \emph{service level objectives} in a wide variety of queueing systems in service, computing, healthcare, and other~domains.

The specific metric we focus on minimizing \emph{tail probabilities}, namely the probabilities a job has latency greater than large thresholds~$t$.
That is, if an M/G/1 scheduling policy induces latency distribution $L$, then the tail probabilities are $\P{L > t}$.
We could equivalently work with \emph{tail quantiles}, namely the $(1 - \epsilon)$th quantiles of $L$ for small values of~$\epsilon$.

Given a fixed threshold~$t$, one might hope to use the Gittins index to minimize $\P{L > t}$ by having a job's Markov chain incur cost~$1$ (meaning, yield a reward of~$-1$) once it has been in the system for time~$t$.
Unfortunately, this type of cost structure cannot be encoded as part of standard MCS (\cref{def:mcs}) or its M/G/1 variant.
The issue is that for the job to incur cost~$1$ after spending time~$t$ in the system, one would need to keep track of the job's time in the system so far as part of its state.
But this quantity changes \emph{even when the job is not in service}, whereas in MCS, a Markov chain only advances when its action is played---meaning, here, that a job changes state only when in service.

There is an extension of MCS, called the \emph{restless bandit} problem \citep{whittle_restless_1988}, in which all Markov chains advance each step, with the selected Markov chain advancing according to a different (typically thought of as \emph{better}) transition kernel and reward function than non-selected Markov chains.
In some cases, the Gittins index can be generalized to the restless bandit setting---in which it is called the \emph{Whittle index}.
This approach has been used to for problems similar to minimizing tail probabilities \citep{yu_deadline_2018, anand_whittles_2018}, but tends to obtain theoretical guarantees that are much weaker than optimality \citep{weber_index_1990}.

Nevertheless, there is a limited way in which ordinary MCS and the Gittins index can handle jobs undergoing some sort of change even when not in service: \emph{discounting}.
Suppose, for instance, that we consider job Markov chains with similar transitions to those in \cref{fig:job_markov_chain_known, fig:job_markov_chain_unknown, fig:job_markov_chain_stages}, but instead of incurring cost~$1$ with all transitions, most transitions incur cost~$0$, with only transitions to the terminal state~$\checkmark$ yielding \emph{reward}~$1$.
Then a job completed at time~$t$ would yield reward $\dscnt^t$, where $\dscnt < 1$ is the discount factor.
Notably, this reward is affected by the \emph{global} time~$t$ that advances every time step---no matter which job is served---which is exactly the type of phenomenon that one typically needs restless bandits to capture.

Translating the above discussion to the M/G/1 setting, it suggests that one might be able to use the Gittins index to maximize a metric like $\E{\dscnt^L}$ for $\dscnt < 1$.
\footnote{Specifically, there are issues to do with arrivals, because one cannot capture the $\E{\dscnt^L}$ metric using time-homogeneous arrivals \citep[Section~3.3]{harlev_gittins_2025}. The work of \citet{yu_strongly_2024} and \citet{harlev_gittins_2025}, which we soon discuss, does resolve these issues, but doing so is among their main technical contributions.}
Unfortunately, while this metric does incentivize completing jobs sooner rather than later, it does not capture tail scheduling well: once a job has accrued large latency, it becomes less and less urgent, because one is already guaranteed to receive a very small reward from it.

\figTailResults

However, recent work by \citet{yu_strongly_2024} and \citet{harlev_gittins_2025} shows that with a small tweak, the metric $\E{\dscnt^L}$ becomes a good proxy for tail probabilities:
\begin{starquote}
    Instead of aiming to \emph{maximize} $\E{\dscnt^L}$ with \emph{discount} factor $\dscnt < 1$, one should aim to \emph{minimize} $\E{\dscnt^L}$ with \emph{inflation} factor $\dscnt > 1$.
\end{starquote}
Indeed, when $\dscnt > 1$, the goal of minimizing $\E{\dscnt^L}$ not only incentivizes completing jobs earlier, but also causes jobs to become more urgent the longer they have waited: the cost to be eventually paid upon completion increases exponentially over time!
Specifically, it is known that under certain light-tail assumptions on the service time distribution, minimizing $\E{\dscnt^L}$ for a carefully chosen value of $\dscnt > 1$ results in \emph{asymptotically minimal tail probabilities}---that is, minimizing the asymptotics of $\P{L > t}$ in the $t \to \infty$ limit in a certain precise sense \citep{boxma_tails_2007, wierman_is_2012}.
This idea led to the first policies that achieve better tail probabilities than simple first-come first-served (FCFS) policies for light-tailed service times \citep{grosof_nudge_2021, van_houdt_stochastic_2022, charlet_tail_2024}.
However, it was viewing the problem as MCS with inflation that led to the discovery of the asymptotically optimal policies---which are instances of the Gittins policy---first for known service times \citep{yu_strongly_2024}, then for general job Markov chains \citep{harlev_gittins_2025}.
See \citet[Appendix~D]{harlev_gittins_2025} for a general account of MCS with inflation.

We conclude that, by using inflation instead of discounting, one can use the Gittins policy to asymptotically minimize tail probabilities $\P{L > t}$ as $t \to \infty$.
Moreover, this translates into state-of-the-art empirical performance for practical values of~$t$:
\*[a.] For known service times, \cref{fig:tail_results} shows that the respective Gittins policy makes a substantial improvement over other baselines.
\* For unknown service times, the Gittins policy is the first policy known to improve upon FCFS's tail asymptotics, so there are no other baselines to compare against.
\*/
We refer the interested reader to \citet{yu_strongly_2024} and \citet{harlev_gittins_2025} for further details.

\section{Conclusion}
\label{sec:conclusion}

We have presented the \emph{Gittins index} (\cref{def:gittins_general}), a tool for solving decision-making problems under uncertainty that require choosing among multiple processes to advance.
The Gittins index yields an optimal policy when these processes are independent---specifically, it solves MDPs that can be expressed as an instance of \emph{Markov chain selection} (MCS, \cref{def:mcs})---and many problems fit directly into this framework (\cref{sec:examples}).
The key idea behind the Gittins index definition is to compare a stochastic action to a deterministic alternative in the \emph{local MDP} (\cref{def:local_mdp}), which continues to make sense in problems beyond MCS.
In various cases, the Gittins index continues to yield strong policies in these more difficult problems (\cref{sec:beyond}).
In particular, we highlighted two practical applications where the Gittins index shows particular promise: Bayesian optimization (\cref{sec:beyond:bayesopt}) and scheduling to minimize tail latency (\cref{sec:beyond:tail}).

\subsection{Additional topics}
\label{sec:conclusion:additional}

There are many Gittins index topics that we did not cover.
For example, we focused on one particular way of defining the Gittins index, but there are actually many equivalent definitions, each of which gives different intuition or insight \citep{gittins_multi-armed_2011, katehakis_multi-armed_1987, whittle_multi-armed_1980}.
We also only briefly touched on proofs, limiting ourselves to a dynamic programming argument in \cref{sec:proof}: however, just as the Gittins index enjoys many definitions, it also enjoys many optimality proofs \citep{weber_gittins_1992, tsitsiklis_short_1994, tsitsiklis_lemma_1986, weiss_branching_1988, dumitriu_playing_2003, bertsimas_conservation_1996, dacre_achievable_1999, whittle_multi-armed_1980}.
See \citet{frostig_four_2016} for an overview of four of the main optimality proof approaches and \cref{sec:proof:discussion} for additional discussion.
Other important topics that we omitted or mentioned only briefly include:
\*[1.] \emph{The history of the Gittins index}, for which we refer the reader to \citet{gittins_multi-armed_2011} and \citet{glazebrook_stochastic_2014}.
We also highlight the latter half of \citet{gittins_bandit_1979}, which records discussion about the Gittins index shortly after its discovery.
\* \emph{Efficiently computing the Gittins index}, which is well understood for finite-state Markov chains \citep{balakrishnan_multi-armed_2014, gast_testing_2023, katehakis_multi-armed_1987}, but, as discussed in \cref{sec:general:computation}, remains a challenge for general infinite-state Markov chains \citep{kelly_multi-armed_1981, kim_robust_2016, farias_optimistic_2022, gittins_multi-armed_2011}.
\* \emph{Formulating the Gittins index in continuous time,} which is conceptually similar, but technically more difficult, than the discrete-time setting we focus on here \citep{mandelbaum_continuous_1987, kaspi_multi-armed_1998, karoui_dynamic_1994, bank_gittins_2007}.
\* \emph{Approximate optimality results} in two settings beyond standard MCS: when one has only approximately computed the Gittins index, and when multiple Markov chains must be played in parallel.
See \citet[Sections~4.10 and~5.7]{gittins_multi-armed_2011} for a treatment of the discounted setting and \citet[Chapters~16 and~17]{scully_new_2022} for a treatment of the queueing setting.
\* \emph{Robust variants of the Gittins index} for settings with misspecified transition kernels \citep{moseley_robust_2025, kim_robust_2016, gupta_markovian_2019, caro_robust_2022, cohen_gittins_2022, scully_uniform_2022}.
\* \emph{Other modern work on the Gittins index}, including applying it to auction design \citep{kleinberg_descending_2016}, adapting it to fairness constraints \citep{aminian_markovian_2023}, applying it to analyze games \citep{clarkson_classical_2023}, better understanding its behavior in queues \citep{scully_when_2025, scully_characterizing_2020, aalto_gittins_2009, aalto_properties_2011, aalto_minimizing_2023}, and proving regret bounds for non-Bayesian bandit settings \citep{lattimore_regret_2016, farias_optimistic_2022}.
\* \emph{Restless bandit problems} \citep{whittle_restless_1988}, in which Markov chains can transition even on time steps when they are not played.
In this setting, a generalization of the Gittins index, called the \emph{Whittle index}, yields a good policy under certain conditions \citep{weber_index_1990}, and similar ideas have recently led to policies that achieve even better performance under more general conditions \citep{verloop_asymptotically_2016, hong_achieving_2024, hong_restless_2023, gast_model_2025, avrachenkov_lagrangian_2024}.
See \citet{nino-mora_markovian_2023} for a recent survey.
\*/

\subsection{Open problems}
\label{sec:conclusion:open_problems}

We conclude by listing several classes of open problems in Gittins indices, some of which have been mentioned throughout our exposition.
The first a comprehensive understanding of \emph{numerical computation} beyond finite-state settings, as discussed in \Cref{sec:general:computation}.
Are there general classes of infinite-state Markov chains for which the Gittins index can be efficiently computed, especially if the state space is high-dimensional?
Can one utilize the very small and structured nature of the local MDP's action space to solve the corresponding dynamic program more efficiently than off-the-shelf approximate dynamic programming methods would allow?
An understanding of these questions would allow Gittins indices to be applied in substantially more complicated settings compared to those which are well understood~today.

The second class of open problems is the \emph{analysis beyond optimality} of the Gittins policy.
We mention several results of this type in \cref{sec:conclusion:additional}, but many open problems remain.
One such problem, which is particularly important for Bayesian optimization (\cref{sec:beyond:bayesopt}), is proving \emph{regret bounds} on the Gittins policy.
For finite-horizon bandits, an important initial step in this direction has been taken by \citet{lattimore_regret_2016} and \citet{farias_optimistic_2022}.
However, at present, even for simple regret in Pandora's box, a corresponding analysis has yet to be developed.
An improved understanding of the Gittins policy's regret, and related quantities appropriate for other setups, could help understand in which situations the key definition is the right approach.

In this context, it is worth noting that compared to approximate optimality, exact optimality is a rigid notion---which, necessarily, captures all phenomena occurring in the problem, including those reflected in constant factors rather than rates.
In contrast, approximate optimality arguments tend to work in greater generality, and can therefore provide a complementary understanding by clarifying which phenomena are specific and which are universal.
Such analyses might therefore reveal properties of the Gittins index that complement what is known from its optimality theory.

A third class of open problems involves understanding \emph{metrics beyond mean performance}.
Here, we have illustrated a Gittins index variant that can be used to minimize tail latency in queueing.
More broadly, many decision-making algorithms admit analogues that seek to perform well in terms of quantile regret, or in terms of high-probability bounds.
We therefore expect there to be decision problems which may appear rather different from the classical Gittins index examples, but which nonetheless can be approached fruitfully using the presented toolkit.

Finally, we believe that, given the scope of generality presented in this tutorial---where \cref{def:mcs} allows for \emph{arbitrary} Markov chains---that Gittins-index-based decision-making should be helpful for a broader set of domains that may otherwise appear to have little to do with queueing and economics, where Gittins indices have traditionally been applied.
Domains where Bayesian optimization is popular, such as chemistry and material design, seem particularly promising: here, Gittins-index-based machinery might allow one to work with more-complex experimental pipelines---or with more-flexible probabilistic models defined by, for instance, diffusion models, rather than traditional Gaussian processes.
To achieve this, developing the aforementioned understanding of numerical methods is a key initial step.

\section*{Acknowledgments}

We thank the anonymous referees for their many helpful comments.
Ziv Scully was supported by the NSF under grant no. CMMI-2307008.
Alexander Terenin was supported by Cornell University, jointly via the Center for Data Science for Enterprise and Society, the College of Engineering, and the Ann S. Bowers College of Computing and Information Science.

% \ifinforms{\bibliographystyle{informs2014}}{\bibliographystyle{ACM-ish}}
% Changing bibliography style because `informs2014` doesn't handle article numbers
\bibliographystyle{ACM-ish}
\bibliography{refs, more-refs}

\clearpage

\begin{appendices}

\section{Optimality of the Gittins index policy}
\label{sec:proof}

In this appendix, we prove \cref{thm:gittins_optimal}, namely optimality of the Gittins policy for MCS (\cref{def:mcs}).
The result is a direct corollary of \cref{thm:v_mcs}, the main result of this appendix, which gives an explicit formula for the value function of MCS.
Our proof, given in \cref{sec:proof:the_actual_proof}, consists of the following steps:
\*[1.] We state the specific assumptions needed for our proof.
Some involve dynamic programming (\cref{asm:bellman_unique_solution}), while others are taken to ease presentation (\cref{asm:simplify}).
\* We define the \emph{surrogate value} of a Markov chain (\cref{def:surval}), a random variable that gives a probabilistic interpretation of the local MDP value function (\cref{lem:v_loc_surval}).
\* We define a guess for the MCS value function by appropriately combining the Markov chains' surrogate prices, then show that it solves the MCS Bellman equation (\cref{thm:v_mcs}).
The rough idea is that our MCS value function guess inherits the respective Bellman inequalities of the Markov chains' local MDPs.
\*/
After the proof, we explain how to extend it from MCS to MCS-$k$ (\cref{sec:proof:mcs-k}), then discuss other proofs of the Gittins policy's optimality from the literature (\cref{sec:proof:discussion}).
Our proof is primarily based on that of \citet{whittle_multi-armed_1980} but uses ideas from other proofs, too.

\subsection{Optimality via dynamic programming using surrogate values}
\label{sec:proof:the_actual_proof}

To begin, define the \emph{Bellman operator of action~$i$}, denoted~$\calB_i$, to be an affine operator that acts on all $f : S_i \to \bbR$ for which $\E_{s_i' \sim p_i(s_i)}{f(s_i')}$ is finite for all $s_i$ by
\[
  (\calB_i f)(s_i) &= r_i(s_i) + \E_{s_i' \sim p_i(s_i)}{f(s_i')}
  .
\]
Using \cref{asm:discount}, by standard theory, the local MDP's optimal value function $V^*_{\loc,i} : S_i \times \bbR \to \bbR$ is well defined and solves the Bellman optimality equation
\[
  \label{eq:bellman_local_mdp}
  V_{\loc,i}^*(s_i;\alpha)
  = \max\gp[\big]{\alpha, \calB_i V_{\loc, i}^*(s_i; \alpha)}
  .
\]
The corresponding Bellman optimality equation for MCS is
\[
  \label{eq:bellman_mcs}
  V_\mcs^*(s_1, \dots, s_n)
  = \max_{i \in \{1, \dots, n\}} \calB_i V_\mcs^*(s_1, \dots, s_n)
  .
\]
where the Bellman operator $\calB_i$ is understood as acting on the function $s_i \mapsto V_\mcs^*(s_1, \dots, s_n)$, and similarly throughout for other expressions using the variables $s_1, \dots, s_n$.
We first make the following assumption, to ensure dynamic programming works the way one expects~it~to.

\begin{assumption}
\label{asm:bellman_unique_solution}
The optimal value function $V_\mcs^* : S_1 \times \cdots \times S_n \to \bbR$ is well defined, and is achieved by at least one policy $\pi^* : S_1 \times \cdots \times S_n \to \{1,\dots,n\}$.
Moreover, a given policy is optimal if and only if its value function satisfies Bellman's optimality equation \cref{eq:bellman_mcs}.
\end{assumption}

We expect \cref{asm:bellman_unique_solution} to follow from \cref{asm:discount}, but do not rigorously check this in order to focus our presentation on Gittins-index-theoretic aspects.
Note, however, that \cref{asm:discount} guarantees that $V^*_{\loc,i}(s_i;\alpha)$ is bounded by the sum of $\alpha$ plus the maximum absolute sum of the rewards, both of which are finite: this implies $V^*_{\loc,i}$ is never infinite and hence well defined.
Using this and finiteness of MCS's action space, one can show that the maximum in \cref{eq:bellman_mcs} is achieved, and therefore defines a policy.
By unrolling this policy's respective Bellman equation and applying the fact that \cref{asm:discount} guarantees either discounting or termination in finite time, one can check that the resulting policy achieves a value of $V^*_{\mcs}$, which is always finite.
In the other direction, any policy which does not satisfy Bellman's optimality equation is suboptimal, as it can be improved by replacing its action in some state with one that maximizes \cref{eq:bellman_mcs}.
A final subtlety worth noting is that, in our setting, \cref{eq:bellman_mcs} \emph{need not} admit a unique solution $V^*_\mcs$, and spurious solutions can occur in practical settings---see \citet[Appendix~D]{scully_optimal_2018} for an example.

For ease of exposition, we make two additional assumptions without loss of generality.

\begin{assumption}
  \label{asm:simplify}
  Both of the following hold:
  \*[subenv]
  \* \label{asm:simplify:undiscounted}
  There is no discounting: $\dscnt = 1$.
  \* \label{asm:simplify:no_free_states}
  All Markov chains have no free states, where a Markov chain state is called \emph{free} if it has non-negative reward and zero probability of transitioning directly into a terminal state.
  \*/
\end{assumption}

The intuition behind \cref{asm:simplify:no_free_states} is in the name: \emph{free, as in beer}.
Free states always give a reward, and never cause the decision problem to end, resulting in only gain with no downside.
So, any reasonable algorithm should always play them whenever they occur.
From a Gittins index viewpoint, this manifests in the form of having to deal with extended-valued functions: this is a technical nuisance, so we will assume without loss of generality that none occur.
To avoid casework, we work in the undiscounted setting, also without loss of~generality.

We now briefly argue why \cref{asm:simplify} can be made without loss of generality.
For~\cref{sub@asm:simplify:undiscounted}, one can use the standard trick of replacing discount factor $\dscnt$ by probability-$(1 - \dscnt)$ transitions from all states to a terminal state.
For~\cref{sub@asm:simplify:no_free_states}, in an MCS instance with free states, one can show that an optimal agent always prioritizes playing Markov chains in free states before those in non-free states.
We can thus eliminate free states from each Markov chain, altering the transition kernels and reward functions in light of the fact that the Markov chain will be advanced until it reaches a non-free state.

Provided there are no free states, the Gittins index is always finite, as shown below.
In fact, it is this property that causes the MCS value function to admit a simple form.

\begin{lemma}
  \label{lem:finite_gittins_index}
  For any state non-terminal $s$ of any Markov chain satisfying \cref{asm:discount}, if $s$ is not free, then $G(s) < \infty$.
\end{lemma}

\begin{proof}
  Consider a Markov chain $(S, \partial S, p, r)$, and define the following quantities for all non-terminal states~$s$:
  \* Let $q(s) = \P_{s' \sim p(s)}{s' \in \partial S}$ be the probability of transitioning directly from $s$ to a terminal state.
  \* Recall that $r(s)$ is the immediate reward of~$s$.
  \* Let random variable $R(s)$ be the total reward received on a random trajectory from $s$ to a terminal state.
  \Cref{asm:discount} tells us $\E{R(s)} < \infty$.
  \*/
  Suppose $s$ is not free, meaning either $q(s) > 0$ or $r(s) < 0$, and consider the $(s, \alpha)$-local MDP
  It suffices to show there exists $\alpha \in \bbR$ such that playing $\st$ is strictly better than playing~$\go$.

  There are two cases to consider, depending on the reason $s$ is not free.
  In both cases, we bound $V_\loc^\goName(s; \alpha)$, the optimal value achievable in the local MDP when playing $\go$ at least once.
  Specifically, we show $V_\loc^\goName(s; \alpha) < \alpha$ for sufficiently large $\alpha \in \bbR$ in both cases, which implies the result.

  When $q(s) > 0$, we apply the bound
  \[
    V_\loc^\goName(s; \alpha)
    \leq \E{R(s)} + (1 - q(s)) \alpha
    .
  \]
  This holds because the maximum expected reward obtainable from the Markov chain is $\E{R(s)}$, and when playing $\go$ first, the probability of ever playing $\st$ is at most $1 - q(s)$.
  Because $\E{R(s)} < \infty$ and $1 - q(s) < 1$, we have $V_\loc^\goName(s; \alpha) < \alpha$ for large enough~$\alpha$.

  When $r(s) < 0$, we apply the bound
  \[
    V_\loc^\goName(s; \alpha)
    \leq r(s) + \E{\max(R(s) - r(s), \alpha)}
    .
  \]
  This holds because the right-hand side is the value that would be achievable if after playing~$\go$ once, the agent learned the Markov chain's full trajectory and clairvoyantly chose between the alternative~$\alpha$ and the remaining trajectory reward $R(s) - r(s)$.
  Because $\E{R(s)} < \infty$, we have
  \[
    \lim_{\alpha \to \infty} \gp[\big]{V_\loc^\goName(s; \alpha) - \alpha}
    \leq r(s) + \lim_{\alpha \to \infty} \E{\max(R(s) - \alpha, 0)}
    = r(s)
    .
  \]
  Because $r(s) < 0$, this means $V_\loc^\goName(s; \alpha) < \alpha$ for large enough~$\alpha$.
\end{proof}

We now introduce the machinery needed to state the MCS value function.
At a high level, this value function is a certain combination of the value functions of the Markov chains' local MDPs.
The specific combination is most easily written and understood using a probabilistic interpretation.
Our next step is therefore to express a Markov chain's local MDP value function probabilistically.
Throughout, let $(S, \partial S, p, r)$ be a Markov chain satisfying \cref{asm:discount, asm:simplify}, and note again by \cref{lem:finite_gittins_index} that all Gittins indices in question are~finite.

\begin{definition}
  \label{def:surval}
  The \emph{surrogate value} of non-terminal state $s \in S \setminus \partial S$, denoted $\surval(s)$, is the random variable equal to the minimum Gittins index on the trajectory from $s$ (inclusive) to a terminal state in $\partial S$ (exclusive).
  More formally, letting $s(0), s(1), \dots$ be a trajectory of the Markov chain with $s(0) = s$, and letting $\tau$ be the hitting time of $\partial S$, we have
  \[
    \label{eq:surval_def}
    \surval(s) = \min_{0 \leq u < \tau} G(s(u))
    .
  \]
\end{definition}

From \cref{eq:surval_def}, it is immediate that a state's surrogate value is at most its Gittins index:
\[
  \label{eq:surval_gittins}
  \P{\surval(s) \leq G(s)} = 1
  .
\]
The following lemma strengthens this by relating the full distribution function of $\surval(s)$ to the local MDP.
In what follows, we use the term \emph{almost all} in the Lebesgue sense.

\begin{lemma}
  \label{lem:v_loc_surval}
  The surrogate value and local MDP's value function satisfy
  \[
    \label{eq:surval_cdf}
    \P{\surval(s) \leq \alpha}
    &= \frac{\d}{\d{\alpha}} V_\loc^*(s; \alpha)
    &
    & \textenv{for almost all } \alpha \in \bbR
    \\
    \label{eq:surval_emax}
    \E{\max(\surval(s), \alpha)}
    &= V_\loc^*(s; \alpha)
    &
    & \textenv{for all } \alpha \in \bbR
    .
  \]
\end{lemma}

\begin{proof}
  We first show \cref{eq:surval_cdf}.
  By the reasoning in the proof of \cref{lem:v_loc}, $\frac{\d}{\d{\alpha}} V_\loc^*(s; \alpha)$, which exists for almost all~$\alpha$ by convexity, is the probability that an optimal policy for the local MDP ever plays~$\st$.
  \footnote{
    One can say more: if the derivative fails to exist, it is because there are multiple optimal policies with different probabilities of playing~$\st$.
    But left and right derivatives still exist in this case, and \citet[Appendix~B.6]{xie_cost-aware_2024} show that they are the minimum and maximum probabilities of playing~$\st$ among optimal~policies.
  }
  But we know that an optimal policy for the local MDP is to play~$\st$ if and only if the current state $s'$ has $G(s') \leq \alpha$.
  The probability this policy ever plays~$\st$ is the probability that some state~$s'$ on the trajectory from $s$ to a terminal state has $G(s') \leq \alpha$.
  By~\cref{eq:surval_def}, this happens if and only if $\surval(s) \leq \alpha$, so the probability is $\P{\surval(s) \leq \alpha}$, as desired.

  Having shown \cref{eq:surval_cdf}, we move on to showing \cref{eq:surval_emax}.
  Recall the definition of $G(s)$ in \cref{eq:gittins_general}:
  \[
    \label{eq:gittins_general_apdx}
    G(s)
    = \sup\curlgp{g \in \bbR : V_\loc^*(s; g) > g}
    = \inf\curlgp{g \in \bbR : V_\loc^*(s; g) = g}
    .
  \]
  Recall also from \cref{lem:finite_gittins_index} that $G(s) < \infty$.
  Combining \cref{eq:surval_cdf, eq:gittins_general_apdx} shows that for all $\alpha \geq G(s)$, we have
  \[
    \E{\max(\surval(s), \alpha)}
    = \alpha
    = V_\loc^*(s; \alpha)
  \]
  so \cref{eq:surval_emax} holds for $\alpha \geq G(s)$.
  By the tail integral formula, for almost all~$\alpha$, we get
  \[
    \frac{\d}{\d{\alpha}} \E{\max(\surval(s), \alpha)} = \P{\surval(s) \leq \alpha}
    .
  \]
  Combining this with \cref{eq:surval_cdf} implies \cref{eq:surval_emax} holds for $\alpha < G(s)$, too.
\end{proof}

We are now ready to prove \cref{thm:gittins_optimal}.
We consider an undiscounted MCS instance with $n$ Markov chains $(S_i, \partial S_i, p_i, r_i)$, all satisfying \cref{asm:discount,asm:bellman_unique_solution,asm:simplify}.
As previously discussed, our approach will be to guess the form of the MCS value function, then show that it satisfies the Bellman equation.

Let $V_\mcs : S_1 \times \dots \times S_n \to \bbR$ be the value function of MCS of this MCS instance.
Because MCS terminates once any of its constituent Markov chains terminates, $V_\mcs^*$ satisfies the terminal condition
\[
  \label{eq:v_mcs_boundary}
  V_\mcs^*(s_1, \dots, s_n) &= 0
  &
  & \text{if } (s_1, \dots, s_n) \in \partial S_\mcs
\]
where the latter is equivalent to $s_i \in \partial S_i$ for some $i$.

To solve the MCS Bellman equation \cref{eq:bellman_mcs}, we need to consider what we know about the operators~$\calB_i$.
The key piece of information is that $\calB_i$ features in the Bellman equation for the local MDP of Markov chain~$i$, as spelled out in \cref{lem:bellman_loc} below.
This suggests that if we define $V_\mcs^*$ by combining the local MDP value functions~$V_{\loc, i}^*$ in a suitable manner, we can control how $V_\mcs^*$ interacts with the Bellman operators~$\calB_i$.
We do this in \cref{thm:v_mcs} below, which relies crucially on the surrogate value interpretation of $V_{\loc, i}^*$ (\cref{lem:v_loc_surval}).

\begin{lemma}
  \label{lem:bellman_loc}
  For all $s_i \in S_i \setminus \partial S_i$ and all $\alpha \in \bbR$, we have
  \[
    V_{\loc, i}^*(s_i; \alpha)
    = \max\gp[\big]{\alpha, \calB_i V_{\loc, i}^*(s_i; \alpha)}
    = \begin{cases}
      \alpha & \text{if } \alpha \geq G_i(s_i) \\
      \calB_i V_{\loc, i}^*(s_i; \alpha) & \text{if } \alpha \leq G_i(s_i)
    \end{cases}
  \]
  with equality between the two branches, namely $\alpha = \calB_i V_{\loc, i}^*(s_i; \alpha)$ if and only if $\alpha = G_i(s_i)$.
\end{lemma}

\begin{proof}
  The first equality is the Bellman equation of the local MDP for Markov chain~$i$ with alternative~$\alpha$, and the rest follows from \cref{def:gittins_general}.
  Specifically, playing~$\st$, which yields value~$\alpha$, is optimal if and only if $\alpha \geq G_i(s_i)$; and playing~$\go$, which yields expected value $\calB_i V_{\loc, i}^*(s_i; \alpha)$, is optimal if and only if $\alpha \leq G_i(s_i)$.
\end{proof}

\begin{theorem}
  \label{thm:v_mcs}
  The MCS optimal value function is
  \[
    \label{eq:v_mcs_guess}
    V_\mcs^*(s_1, \dots, s_n)
    = \begin{dcases}
        \E[\bigg]{\max_{i \in \{1, \dots, n\}} \surval_i(s_i)} & \text{if } s_i \in S_i \setminus \partial S_i \text{ for all } i
        \\
        0 & \text{otherwise.}
    \end{dcases}
  \]
  Moreover, the maximizing actions in the Bellman equation \cref{eq:bellman_mcs} are those with maximal Gittins index:
  \[
    \argmax_{i \in \{1, \dots, n\}} \calB_i V_\mcs^*(s_1, \dots, s_n)
    = \argmax_{i \in \{1, \dots, n\}} G_i(s_i)
  \]
  and the value $V^*_\mcs$ is achieved by the Gittins policy.
\end{theorem}

\begin{proof}
  For the purposes of this proof, let us take $V_\mcs^*$ to be the function \emph{defined} by the expression \cref{eq:v_mcs_guess}, which we emphasize is not assumed to be the optimal value.
  By \cref{asm:bellman_unique_solution}, if we can show that this ansatz satisfies the Bellman equation \cref{eq:bellman_mcs} and is the value of some policy $\pi^*$, then $V_\mcs^*$ is indeed the true optimal value~function.

  Suppose $s_i \in S_i \setminus \partial S_i$ for all~$i$.
  It suffices to show that for all~$i$,
  \[
    \calB_i V_\mcs^* (s_1, \dots, s_n) \leq V_\mcs^*(s_1, \dots, s_n)
  \]
  with equality if and only if Markov chain~$i$ has maximal Gittins index, meaning $G_i(s_i) \geq G_j(s_j)$ for all $j \neq i$.

  Below, to reduce clutter, we shorten $G_i(s_i)$ to~$G_i$ and shorten $\surval_i(s_i)$ to~$\surval_i$.
  Recall throughout that surrogate values of different Markov chains are mutually independent.
  Let
  \[
    \label{eq:surval_others}
    \surval_{\neq i}
    = \max_{j \neq i} \surval_j
    .
  \]
  We first check that $\calB_i V_\mcs^*$ is well defined: this follows by the fact that rewards are bounded in absolute value (\cref{asm:discount}).
  Next, using \cref{lem:v_loc_surval}, we can write the proposed MCS value function in terms of $\surval_{\neq i}$:
  \[
    \label{eq:v_mcs_local_perspective}
    V_\mcs^*(s_1, \dots, s_n)
    = \E{\max(\surval_i, \surval_{\neq i})}
    = \E{V_{\loc, i}^*(s_i, \surval_{\neq i})}
    .
  \]
  One intuition is that \cref{eq:v_mcs_local_perspective} shows the perspective of $i$ on the MCS instance, where the value of playing any action other than~$i$ has been summarized by the random variable~$\surval_{\neq i}$.
  Applying the (affine) Bellman operator~$\calB_i$ and using \cref{lem:bellman_loc} yields
  \[
    \calB_i V_\mcs^*(s_1, \dots, s_n)
    & \label{eq:v_mcs:b_e}
    = \calB_i \E{V_{\loc, i}^*(s_i; \surval_{\neq i})}
    \\ & \label{eq:v_mcs:e_b}
    = \E{\calB_i V_{\loc, i}^*(s_i; \surval_{\neq i})}
    \\ &
    \leq \E{V_{\loc, i}^*(s_i; \surval_{\neq i})}
    \\ & \label{eq:v_mcs:bellman_ineq_conclusion}
    = V_\mcs^*(s_1, \dots, s_n)
  \]
  with equality if and only if $\P{\surval_{\neq i} \leq G_i} = 1$.
  Here, changing the order of expectations when going from \cref{eq:v_mcs:b_e} to \cref{eq:v_mcs:e_b} follows by Fubini's Theorem, using absolute integrability of the reward sum (\cref{asm:discount}).
  One subtlety is that while we have assumed $s_i$ is non-terminal, it might be that the next state $s_i' \sim p_i(s_i)$, which is used implicitly when we apply the Bellman operator, is terminal.
  To handle this, one can check that the left- and right-hand expressions in \cref{eq:v_mcs_local_perspective} are still equal, namely both~$0$, when $s_i$ is terminal.

  It remains only to show that $\P{\surval_{\neq i} \leq G_i} = 1$ if and only if $G_i \geq G_j$ for all $j \neq i$.
  The \emph{if} direction follows from \cref{eq:surval_gittins, eq:surval_others}, which together imply
  \[
    \surval_{\neq i} = \max_{j \neq i} \surval_j \leq \max_{j \neq i} G_j \leq G_i
  \]
  with probability~$1$.
  For the \emph{only if} direction, because $\surval_{\neq i} \geq \surval_j$ for all~$j$, it suffices to show that if $G_i < G_j$ for some~$j$, then $\P{\surval_j \leq G_i} < 1$.
  \Cref{lem:v_loc, lem:v_loc_surval} together imply that for all $\alpha \in \bbR$, the following three expressions are equivalent:
  \[
    \P{\surval_j \leq \alpha} &< 1
    &
    \frac{\d}{\d{\alpha}} V_\loc^*(s_j; \alpha) &< 1
    &
    \alpha &< G_j
    .
  \]
  The desired statement follows by plugging in $\alpha = G_i$.

  Finally, we argue that the value achieved by the Gittins policy is indeed given by~$V^*_\mcs$: in some sense, this ensures that $V^*_\mcs$ is not a spurious solution to the Bellman equation.
  Denote the Gittins policy, under an arbitrary tie-breaking rule, by $\pi^* : S_1 \times \dots \times S_n \to \{1,\dots,n\}$, where, as with the notation used before, we do not assume optimality.
  Let $s(0), \dots, s(\tau)$, without subscripts, be the trajectory of the MCS state vector under~$\pi$, where $s(0) = s = (s_1, \dots, s_n)$ is the initial state, and~$\tau$ is the time when MCS terminates, namely the hitting time of $\partial S_\mcs$.
  By the preceding argument, we know that
  \[
  V^*_\mcs(s) = \max_{i \in \{1, \dots, n\}} \calB_i V^*_\mcs(s) = \calB_{\pi^*(s)} V^*_\mcs(s)
  .
  \]
  For every finite $T \geq 1$, iterating the above expression and using the fact that $V^*_\mcs(s(\tau)) = 0$ by \cref{eq:v_mcs_boundary} gives
  \[
  V^*_\mcs(s) = \E[\Bigg]{\smashoperator[r]{\sum_{t=0}^{\min(T, \tau) - 1}} r_{\pi^*(s(t))}(s_{\pi^*(s(t))}) + V^*_\mcs(s(\min(T, \tau)))}
  .
  \]
  Taking the $T \to \infty$ limit and applying dominated convergence via \cref{asm:discount} yields
  \[
  V^*_\mcs(s) = \E[\Bigg]{\sum_{t=0}^{\tau - 1} r_{\pi^*(s(t))}(s_{\pi^*(s(t))})}
  .
  \]
  The right-hand side is exactly the value achieved by the Gittins policy, as desired.
\end{proof}

\begin{remark}
  \label{rmk:mdps_failure}
  As discussed in \cref{sec:beyond:optional_inspection}, one cannot in general extend \cref{thm:gittins_optimal, thm:v_mcs} from MCS to MDP selection.
  However, part of the argument still goes through:
  \*[subenv] On one hand, the \emph{Bellman inequality}, namely \crefrange{eq:v_mcs:b_e}{eq:v_mcs:bellman_ineq_conclusion}, continues to hold.
  This implies the value function guess in \cref{eq:v_mcs_guess} is an upper bound for the MDP selection value function, a fact which plays a crucial role in many approximation results for MDP selection \citep{brown_optimal_2013, scully_local_2024, chawla_combinatorial_2024, aouad_pandoras_2020, beyhaghi_recent_2023, bowers_prophet_2025}.
  \* On the other hand, the \emph{Bellman equation} generally fails: there need not exist an action that makes the Bellman inequality tight.
  The issue is that in \cref{eq:v_mcs:e_b}, the random realization of $\surval_{\neq i}$ might influence which action is optimal, but we have to choose an action without knowledge of~$\surval_{\neq i}$---recall that, here, Bellman operators are now parameterized by a pair $(i, a)$, where $a$ is an action in MDP~$i$.
  One of the few cases when the Bellman equation holds despite this obstacle is when all the MDPs satisfy the Whittle condition (see \cref{sec:beyond:optional_inspection:progress}).
  \*/
\end{remark}

\subsection{Generalization: finishing multiple Markov chains}
\label{sec:proof:mcs-k}

We now sketch how the statement and proof of \cref{thm:v_mcs} can be generalized from MCS to MCS-$k$ (\cref{def:mcs-k}).
The overall approach we take is the similar to that of \citet{scully_local_2024}, although that work considers just the Pandora's box setting and its optional inspection variant.
See \citet{singla_price_2018} and \citet{gupta_markovian_2019} for alternative proofs.
One can extend the argument below to variants of MCS-$k$ involving combinatorial constraints, but we treat only ordinary MCS-$k$ for ease of presentation.

The main difference here is that instead of the MCS value function involving the maximum single surrogate value, the argument involves the \emph{sum of the $k$ greatest surrogate values}.
Specifically, when there are $k$ items still to be selected, we have
\[
  \label{eq:v_mcs-k}
  V_\mcsk^*(s_1, \dots, s_n)
  = \E*{\max_{\substack{I \subseteq \{1, \dots, n\}\\|I| = k}} \sum_{i \in I} \surval_i(s_i)}
  .
\]
More generally, when exactly $\ell \leq k$ of the Markov chains are in terminal states, the same formula holds, except we use only the $k - \ell$ greatest surrogate values, meaning we replace $|I| = k$ by $|I| = k - \ell$.
In particular, because the empty sum is~$0$, this gives a boundary condition $V_\mcsk(s_1, \dots, s_n) = 0$ when there are $k$ Markov chains in terminal states.

To show that \cref{eq:v_mcs-k} gives a solution to the MCS-$k$ Bellman equation, we use an analogue of \cref{eq:v_mcs_local_perspective}, expressing the MCS-$k$ value function in terms of the local MDP's value function.
To reduce clutter, we again shorten $G_i(s_i)$ to~$G_i$ and $\surval_i(s_i)$ to~$\surval_i$, and we assume without loss of generality that none of the Markov chains are in terminal states.
Letting
\[
  \surval_{\text{with\,} i} &= \max_{\substack{I \subseteq \{1, \dots, n\} \setminus \{i\}\\|J| = k - 1}} \sum_{j \in J} \surval_j
  &
  \surval_{\text{without\,} i} &= \max_{\substack{I \subseteq \{1, \dots, n\} \setminus \{i\}\\|J| = k}} \sum_{j \in J} \surval_j
\]
we can rewrite \cref{eq:v_mcs-k} as
\[
  V_\mcsk^*(s_1, \dots, s_n)
  &
  = \E{\max(\surval_i + \surval_{\text{with\,} i}, \surval_{\text{without\,} i})}
  \\ &
  = \E{\max(\surval_i, \surval_{\text{without\,} i} - \surval_{\text{with\,} i})} + \E{\surval_{\text{with\,} i}}
  \\ &
  \label{eq:v_mcs-k_local_perspective}
  = \E{V_{\loc, i}^*(s_i; \surval_{\text{without\,} i} - \surval_{\text{with\,} i})} + \E{\surval_{\text{with\,} i}}
  .
\]
Just as \cref{eq:v_mcs_local_perspective} can be thought of as $i$'s perspective on MCS, \cref{eq:v_mcs-k_local_perspective} can be thought of as $i$'s perspective on MCS-$k$.
Because neither $\surval_{\text{with\,} i}$ nor $\surval_{\text{without\,} i}$ depend on~$s_i$, applying the Bellman operator and then reasoning similarly to the end of proof of \cref{thm:v_mcs} shows the that Bellman equation holds, with playing~$i$ being optimal if and only if $G_i$ is among the $k$ greatest Gittins indices.

\subsection{Comparison to other optimality proofs}
\label{sec:proof:discussion}

There are many proofs of the optimality of the Gittins policy in the literature, taking a variety of approaches and covering a variety of settings.
We discuss just a few approaches here, referring the reader to \citet[Section~2.12]{gittins_multi-armed_2011} for a more comprehensive history.

To the best of our knowledge, there is no single theorem that unifies all of the known optimality results, particularly when infinite spaces, continuous time, or inflation (as in \cref{sec:beyond:tail}) are involved.
We consider providing such a unifying account to be an open~problem.

For settings where Markov chains have finitely many states, a common approach is to use \emph{induction on the number of states}.
See \citet{tsitsiklis_short_1994} for a particularly accessible proof of this form.
Roughly speaking, these proofs show that the state of maximal index should be prioritized over all others, then reduce the Markov chain involved by removing that state.
One advantage of this inductive approach is that it easily generalizes to branching bandits \citep{weiss_branching_1988}.
Another advantage is that the proofs are elegant and elementary.
On the other hand, a key difficulty is that they are hard to generalize to infinite state spaces.

The proof we give above is based primarily on that of \citet{whittle_multi-armed_1980}, which is the first \emph{dynamic programming} proof of the Gittins policy's optimality.
In particular, \citet{whittle_multi-armed_1980} discovered the form of the MCS value function as a combination of all the Markov chains' local MDP value functions, though without the probabilistic interpretation of \cref{thm:v_mcs}.
Unlike the inductive approach, the dynamic programming approach works essentially whenever dynamic programming itself works---which our argument took as assumption.

An advantage of the dynamic programming approach is that it can be easily extended to yield results about MDP selection (\cref{sec:beyond:optional_inspection:progress}), including optimality when the MDPs satisfy the Whittle condition \citep{whittle_multi-armed_1980, glazebrook_sufficient_1982} and approximation results when they do not \citep{brown_optimal_2013, aouad_pandoras_2020, chawla_combinatorial_2024, scully_local_2024}.
Even more generally, in the setting where independent Markov chains are replaced by a generalization called \emph{interleaved filtrations}, essentially the same construction still works \citep{mandelbaum_discrete_1986}, and extends to continuous time \citep{mandelbaum_continuous_1987, kaspi_multi-armed_1998, karoui_dynamic_1994, bank_gittins_2007}.

Our proof is also influenced by proofs based on an \emph{economic argument}, as pioneered by \citet{weber_gittins_1992} for the discounted setting and later replicated in undiscounted settings \citep{dumitriu_playing_2003, kleinberg_descending_2016, gupta_markovian_2019, bowers_matching_2024, chawla_combinatorial_2024}.
In the undiscounted setting, the argument consists of three main steps:
\*[1.] For each Markov chain~$i$, define a random variable called its \emph{surrogate value}, denoted~$\surval_i$.
This is exactly our \cref{def:surval}.
\* Show that, in MCS, the expected value achieved by any policy~$\pi$ is at most $\E{\surval_{i_\pi}}$, where $i_\pi$ is the identifier of the Markov chain that $\pi$ finishes, noting that there is always exactly one such Markov chain under \cref{asm:discount:non_discounted}.
\* Show that the above inequality is in fact an equality when $\pi$ is the Gittins policy.
\* Finally, observe that the Gittins policy \emph{always} finishes the Markov chain of maximal surrogate value, and thus always obtains surrogate value $\max_{i \in \{1, \dots, n\}} \surval_i$.
\*/
The economic argument thus gives another interpretation of the value function we derive in \cref{thm:v_mcs}: the expected value achieved by Gittins is $\E{\max_{i \in \{1, \dots, n\}} \surval_i}$, and every other policy~$\pi$ achieves at most $\E{\surval_{i_\pi}}$.
One can therefore loosely think of $\surval_i$ as a kind of amortized value for each chain, hence the name \emph{surrogate value} for it.

Finally, two more proof techniques, which are related to each other, are those based on the \emph{achievable region} approach \citep{bertsimas_achievable_1995, bertsimas_conservation_1996, bertsimas_optimization_1999, bertsimas_optimization_1999a, dacre_achievable_1999} and the \emph{WINE (work integral number equality)} queueing identity \citep{scully_new_2022, scully_gittins_2021, scully_gittins_2020}.
The main appeal of these approaches is that they can also be used to prove guarantees on the approximate optimality of approximate index policies, and to prove performance bounds for the multiserver case where multiple Markov chains are advanced at every time step \citep{bertsimas_optimization_1999a, glazebrook_parallel_2001, glazebrook_analysis_2003}.
We refer the reader to \citet{dacre_achievable_1999} for a primer on the achievable region approach and to \citet{scully_new_2022} for a primer on WINE.
There is not yet a full account of the relationship between these approaches, but see \citet[Section~2.2.3]{scully_new_2022} for some initial discussion.
A full unification would likely require a version of the achievable region method that works with measure-valued linear programs.

\end{appendices}

\end{document}